\documentclass[a4paper,12pt]{article}

\usepackage{amsmath,amssymb,amsthm}

 \usepackage{verbatim} 
\usepackage{xcolor,tikz}
\usepackage[vmargin=2cm,hmargin=2cm]{geometry}
\usepackage[hidelinks]{hyperref}
\usepackage{url}
\usepackage[makeroom]{cancel}
\usepackage[labelfont=bf]{caption}

\newcommand*{\N}{\mathbb{N}}
\newcommand*{\R}{\mathbb{R}}

\newcommand*{\C}{\mathbb{C}}
\newcommand*{\E}{\mathbb{E}}
\newcommand{\iu}{\mathrm{i}} 
\newcommand*{\A}{A}
\newcommand*{\B}{B}

\DeclareMathOperator{\trace}{Tr}
\DeclareMathOperator{\Real}{Re}
\DeclareMathOperator{\sgn}{sgn}

\newcommand*\xbar[1]{%
  \,\hbox{%
    \vbox{%
      \hrule height 0.1pt 
      \kern0.4ex
      \hbox{%
        \kern-0.1em
        \ensuremath{#1}%
        \kern-0.1em
      }%
    }%
  }\,%
} 

\providecommand{\keywords}[1]
{
  \small	
   \textbf{\textit{Keywords:}} #1
}

\providecommand{\MSC}[1]
{
  \small	
  \textbf{\textit{2000 MSC:}} #1
}

\theoremstyle{plain}
\newtheorem{theo}{Theorem}[section]

\newtheorem{lemma}[theo]{Lemma}
\newtheorem{propo}[theo]{Proposition}

\newtheorem{setting}[theo]{Setting}
\newtheorem{remark}{Remark}

\theoremstyle{definition}

\theoremstyle{remark}
\newtheorem*{note}{Note}

\title{The Collective Dynamics of a Stochastic Port-Hamiltonian Self-Driven Agent Model in One Dimension}
\author{Matthias Ehrhardt and Thomas Kruse\\
Applied and Computational Mathematics\\
University of Wuppertal, Germany\\
  \and
Antoine Tordeux\\
	Traffic Safety and Reliability\\University of Wuppertal, Germany
	}
\date{\today}

\begin{document}

\maketitle

\begin{abstract}
The collective motion of self-driven agents is a phenomenon of great interest in interacting particle systems. 
In this paper, we develop and analyze a model of agent motion in one dimension with periodic boundaries using a stochastic port-Hamiltonian system (PHS). 
The interaction model is symmetric and based on nearest neighbors. 
The distance-based terms and kinematic relations correspond to the skew-symmetric Hamiltonian structure of the PHS, 
while the velocity difference terms make the system dissipative. 
The originality of the approaches lies in the stochastic noise that plays the role of an external input.
It turns out that the stochastic PHS with a quadratic potential is an Ornstein-Uhlenbeck process, for which we explicitly determine the distribution for any time $t\ge 0$ and in the limit $t\to\infty$.

We characterize the collective motion by showing that the agents' mean velocity is Brownian noise whose fluctuations increase with time, 
while the variance of the agents' velocities and distances, which quantify the coordination of the agents' motion, converge. 
The motion model does not specify a preferred direction of motion. 
However, assuming a equilibrium uniform starting configuration, the results show that the noise triggers rapidly coordinated agent motion determined by the Brownian behavior of the mean velocity. 
Interestingly, simulation results show that some theoretical properties obtained with the Ornstein-Uhlenbeck process also hold for the nonlinear model with general interaction potential. 
\end{abstract}


\MSC{76A30, 
     82C22, 
     60H10, 
     37H30  
}



\section{Introduction}

Collective motion of self-driven agents is the spontaneous formation of coordinated behavior in a common direction of motion. 
Collective motion and swarming behavior can be observed in various systems of living organisms (schools of fish, flocks of birds, herds of animals, colonies of bacteria), in crowds and road traffic, or non-living active systems, such as microswimmers and other self-driven particle systems (see, e.g., the reviews \cite{vicsek2012collective,marchetti2013hydrodynamics,shaebani2020computational}). 
The large variety of systems in which collective motion can be observed is fascinating. 
It suggests that the phenomenon is universal. 
It is then interesting to formulate parsimonious motion models that reproduce the spontaneous coordination of the dynamics. 
This field of research is largely developed in statistical physics and the physics of active matter \cite{acebron2005kuramoto,ballerini2008interaction,Gautrais2012,ramaswamy2017active,shaebani2020computational,keta2022disordered}. 

The collective motion of a multi-agent system is generally characterized using the ensemble mean velocity of the agents as an order parameter \cite{vicsek2012collective}.
One of the most famous microscopic models describing collective motion is the Vicsek model introduced by Tam{\`a}s Vicsek et al.\ in the 1990s \cite{vicsek1995novel} and its extensions \cite{chate2008modeling,degond2014hydrodynamics,moreno2020collective}.
This model shows a phase transition from a disordered state to large-scale ordered motion in one \cite{czirok1999collective} and two dimensions \cite{vicsek1995novel}. 
Nowadays, many studies report on phase transitions to collective motion using multi-agent and self-driven particle systems, including various topological and metric interaction fields and noise models \cite{barberis2019phase,nemoto2019optimizing,grossmann2020particle,martin2021fluctuation,de2022collective}.   
%
In traffic engineering, the study of long-term dynamics, the so-called collective behavior, and its control is a current field of research. 
Initial work from the 1950s \cite{chandler1958traffic, gazis1961nonlinear,herman1959traffic,pipes1953operational} showed that nonlinear follow-the-leader models can exhibit severe stability problems.
However, this topic has currently returned to the focus of public interest in connection with automated driving.
For example, driving assistance systems with \textit{adaptive cruise control} (ACC) can lead to unstable group dynamics of vehicles \cite{ciuffo2021requiem,gunter2020commercially, makridis2021openacc,stern2018dissipation}.
 Thus, it is of great interest to make future driving assistance systems both efficient and robust to disturbances using special stabilization techniques based mostly on anticipation approaches \cite{khound2023extending,treiber2006delays,wang2019effect}. 

In the 1980s, pioneering work by Arjan van der Schaft introduced \textit{port-Hamiltonian systems} (PHS) for modeling nonlinear physical systems \cite{van2007port, van2014port}.
In contrast to the usual Hamiltonian systems, PHS, in addition to describing the Hamiltonian dynamics through the input/output ports, 
also allow control and external factors to be taken into account and enable direct calculation of the system output. 
Similarly, systems from different physical domains (so-called multi-physics systems) can be formulated as PHS and the proper coupling of PHS systems again yields a PHS system \cite{rashad2020twenty}.
It is this functional structure of the PHS, which mitigates the modeling between conserved quantities, dissipation, input, and output, that is an extremely useful representation of many systems.

There are several ways to impart randomness to PHS.
Several works in the literature use explicit, finite-dimensional input-state-output PHSs and study the impacts of white noise perturbations resulting in the state dynamics 
\begin{equation*}
  dX(t)=\Bigl[\bigl( J(X(t))-R(X(t)) \bigr)\nabla H(X(t))+g(X(t))\,u(t)\Bigr]\,dt + \sigma(X(t))\,dW(t)
\end{equation*}
and output $y(t)=g^\top(X(t))\nabla H(X(t))$. 
This approach was put forward in \cite{satoh2012passivity} and was subsequently extended in, e.g., 
\cite{cordoni2021stabilization, fang2017stabilization, lamoline2017stochastic, satoh2017input}. 
A key observation for this approach is that,
although the infinitesimal influence of noise on the state is in expectation zero, 
it nevertheless increases the mean energy of the system at the rate $\operatorname{tr}(\sigma^\top (\operatorname{Hess}H) \sigma )/2$
and therefore these stochastic expansions are no longer inherently passive.
First proposals for incorporating randomness into PHS through the implicit formalism of Dirac structures were recently made in \cite{cordoni2019stochastic} and \cite{lamoline2022dirac}.

In this paper, we propose a stochastic motion model of agents that can be formulated using a stochastic port Hamiltonian system. More precisely, we consider the motion of $N\in \N$ agents on a ring of length $L>0$. In our model the infinitesimal change of velocity of agent $n\in\{1,\ldots,N\}$ depends linearly on the velocity difference to the two direct neighbors. In particular, agent $n$ accelerates if her speed is smaller than the mean speed of the two direct neighbors and slows down otherwise. Moreover, her infinitesimal change of velocity depends on the distance to the two direct neighbors. We introduce a convex potential $U\colon \R \to [0,\infty)$ that dictates how agent $n$ reacts to the distances. The bigger the distance to the agent in front and the smaller the distance to the follower, the higher the acceleration and vice versa. In particular, the agent dynamics are symmetric, with no preferred direction of motion. This deterministic law of motion is perturbed by agent-individual white noise processes.

We show that our model suits into the stochastic pH framework outlined above and analyze its Hamiltonian behavior. In the case of a quadratic potential we perform an explicit analysis on the distributional properties of the system. In particular, we characterize its long-time limit in closed form using an eigendecomposition of the system matrix and results on so-called Dowker's sums. We find that the ensemble mean velocity diverges while the ensemble variances of the agent velocities and distances converge to a steady-state distributions - thereby establishing collective motion.  
In contrast to classical approaches, this collective motion is purely noise-induced and does not result from a phase transition. 
More precisely, the motion model being linear in the case of a quadratic potential, it allows the collective motion to be analyzed explicitly. 
Interestingly, the simulation experiments show that some of the theoretical results seem to remain valid for general potentials and nonlinear interaction terms.

The paper is organized as follows.
In Section~\ref{sec:model} we introduce our agent motion model under consideration jointly with its formulation as a port-Hamiltonian system.
In Section~\ref{sec:OUprocess}
we consider the case of a quadratic potential for which the system is an Ornstein-Uhlenbeck process.
Finally, we present in Section~\ref{sec:simulation} some illustrative simulation results that support our theoretical findings.

\paragraph{General notations:}
Throughout this article we use the following notations: For $N\in \N$ the vector
$\mathbf{1}=(1,1,\dots,1)^\top\in \R^N$ denotes the vector consisting of ones and $I\in \R^{N\times N}$ denotes the $N$-dimensional identity matrix. 
We denote by $\iu\in \C$ the imaginary unit. 
For $a,b\in \R$ we denote by 
$\xbar{a+\iu b}=a-\iu b$ 
the complex conjugate and by $\Real(a+\iu b)=a$ 
the real part of the complex number $a+\iu b\in \C$. 
For a matrix $A\in \C^{N\times N}$ we denote by $A^*=\bar{A}^\top$ its complex conjugate.

\section{Definition of the Model and its Hamiltonian Behavior}\label{sec:model}
Let us first start with defining the model under consideration.

\subsection{Notations}
We consider $N\in \{3,4,\dots,\}$ agents on a segment of length $L$ with periodic boundaries.
We denote 
\begin{equation*}
    q(t)=\bigl(q_n(t)\bigr)_{n=1}^N\in\R^N,\qquad t\in [0,\infty),
\end{equation*}
and 
\begin{equation*}
    p(t)=\bigl(p_n(t)\bigr)_{n=1}^N\in\R^N,\qquad t\in [0,\infty),
\end{equation*}
the positions and velocities of the agents at time $t$, respectively.
We assume that the positions $q(t)$ and the velocities $p(t)$ of the $N$ agents at time $t=0$ are known, 
\begin{equation*}
    p(0)=p_0\in\R^N,\qquad 
    q(0)=q_0=(q_n^0)_{n=1}^N\in\R^N,
\end{equation*}
and that the positions of the agents are initially ordered by their indices, i.e.,
\begin{equation}\label{eq:ini_order}
     0\le q_1^0\le q_2^0\le\ldots\le q_N^0\le L.
\end{equation}
\begin{note}
We systematically use in the following the index $n+1$ for the nearest neighbor on the right and $n-1$ for the nearest neighbor on the left. 
Note that the right neighbor of the $N$th agent is the first agent, 
i.e., $n+1=1$ if $n=N$, and conversely, the left neighbor of the 1st agent is the $N$th agent, i.e., $n-1=N$ if $n=1$.
\end{note}

The distances of the agents to their immediate right neighbors
\begin{equation*}
    Q(t)=(Q_n(t))_{n=1}^N\in\R^N,\qquad t\in [0,\infty),
\end{equation*}
   are given by
\begin{equation}\label{eq:def_dist}
\Bigg\{~\begin{aligned}
    &Q_n(t)=q_{n+1}(t)-q_n(t),&& n\in\{1,\dots,N-1\},\\
    &Q_N(t)=L+q_1(t)-q_N(t).
\end{aligned}
\end{equation}
The distance to the left is $Q_N$ for the first agent and $Q_{n-1}$ for the $n$th agent, $n\in\{2,\dots,N\}$. The index order of the agents at time zero makes the initial distance positive 
\begin{equation*}
    Q(0)=Q_0\in[0,\infty)^N.    
\end{equation*}

\subsection{Agent Motion Model}
To formulate our stochastic motion model, we introduce a probability space $(\Omega, \mathcal{F}, \mathbb{P})$ 
on which there exists an $N$-dimensional standard Brownian motion 
\begin{equation*}
    W=\bigl(W_n\bigr)_{n=1}^N\colon [0,\infty) \times \Omega \to \R^N.
\end{equation*}
The motion model reads for the $n$-th agent at time $ t\in [0,\infty)$
\begin{equation}\label{eq:modn}
    \begin{cases}
        dQ_n(t)&=\bigl(p_{n+1}(t)-p_n(t)\bigr)\,dt,\qquad\qquad\qquad\qquad\qquad\qquad\qquad Q(0)=Q_0,\\
        dp_n(t)&=\bigl(U'(Q_n(t))-U'(Q_{n-1}(t))\bigr)\,dt\\
        &\qquad+\beta\bigl(p_{n+1}(t)-2p_n(t)+p_{n-1}(t)\bigr)\,dt+\sigma \,dW_n(t),  \qquad p(0)=p_0, 
    \end{cases}
\end{equation}
with $\beta\in (0,\infty)$ a dissipation rate, $\sigma\in\R$ the noise volatility and $U'$ the derivative of a convex potential $U\in C^1(\R,[0,\infty))$.  
Stronger regularity assumptions may be necessary to guarantee the existence of the solution.
A common choice is the quadratic functional $U(x)=(\alpha x)^2/2$, $x\in \R$, for some $\alpha \in (0,\infty)$. 

\begin{remark}
    Equation \eqref{eq:modn} describes the motion model in a Hamiltonian fashion. 
    An asymmetric model with the same characteristics has recently been introduced for modeling stop-and-go waves in traffic flow \cite{ruediger2022stability}.
    The first line in \eqref{eq:modn} is simply the differential version of \eqref{eq:def_dist}. 
    The second line in \eqref{eq:modn} models the stochastic law of motion of the $N$ agents. 
    The acceleration $dp_n$ of agent $n$ depends on the velocities of her direct neighbors and the distances to her direct neighbors. 
    Let us first consider the dependence on the velocities. 
    The acceleration $dp_n$ of agent $n$ increases linearly in the difference $p_{n+1}-p_n$ between her right neighbor's velocity and her own velocity and decreases linearly in the difference $p_{n}-p_{n-1}$ between her own velocity and her left neighbor's velocity (both with slope $\beta> 0$). 
    In particular, if her velocity coincides with the mean velocity of her two neighbors (i.e., $p_n(t)=(p_{n+1}(t)-p_n(t))/2)$, then the velocity term does not contribute to the acceleration of agent $n$. 
    Concerning the dependence on the distance to the direct neighbors, we remark that the acceleration increases the higher the distance $Q_n$ to the right neighbor and decreases the higher the distance $Q_{n-1}$ to the left neighbor. 
    The precise dependence on the distances is dictated by the derivative of the convex potential $U$. 
    If agent $n$ is located exactly in the middle between its neighbors (i.e., $Q_n(t)=Q_{n-1}(t)$), then the distance term does not contribute to the acceleration of agent $n$. 
    From the description so far, we see that this deterministic system is in equilibrium if all agents move at the same speed and have equidistant positions. 
    The last term in the second line of \eqref{eq:modn}, however, introduces a stochastic perturbation by white noise which brings the system out of equilibrium.  
\end{remark}

\begin{remark}
Note that the motion model is symmetric: the interaction model with neighbors is identical whether the agent moves to the right (velocity positive) or to the left (velocity negative). 
There is no preferred direction of motion. 
\end{remark}

\begin{remark}\label{rem:average_distance}
The velocity dynamics depends only on the relative positions (distance) and the relative velocities to the two nearest neighbors. 
It is therefore more convenient to represent the system by the distance and velocity variables $(Q,p)$ instead of $(q,p)$. 
Note that in the $(Q,p)$ representation, the relation
\begin{equation*}
   \sum_{n=1}^N Q_n(t)=L,
\end{equation*}
holds at any time $t\in[0,\infty)$ because of the periodic boundaries. 
In particular, for any $n\in \{1,\dots,N\}$ the distance $Q_n$ can be derived from 
the other distances $Q_l$, $l\in \{1,\dots,N\}\setminus\{n\}$, 
and thus the linear equations describing the dynamics of $Q$ in \eqref{eq:modn} are linearly dependent.
\end{remark}

\begin{remark}\label{rem:average_velocity}
The ensemble's mean velocity
\begin{equation}\label{eq:average_velocity}
     \xbar{p}(t)
    =\frac{1}{N}\sum_{n=1}^N p_n(t), \qquad t\in [0,\infty),
\end{equation}
is a Brownian motion with variance $\sigma^2/N$ for any potential function $U$. 
In fact, thanks to the telescopic form of the model~\eqref{eq:modn} and the periodic boundaries, we have
\begin{equation*}\begin{aligned}
   d\xbar{p}(t)=\frac{1}{N}\sum_{n=1}^N dp_n(t)
   =\frac{\sigma}N\sum_{n=1}^N dW_n(t), \qquad t\in [0,\infty).
\end{aligned}
\end{equation*}
\end{remark}

\begin{remark}
We note that within our model it is possible that the initial ordering \eqref{eq:ini_order} is not preserved at all times and that collisions happen, i.e., for each $t\ge 0$ there exists with positive probability $n\in \{1,\ldots,N\}$ such that $Q_n(t)\le 0$.
This is due to the fact that each agents' velocity is driven by an individual Brownian motion. 
Incorporating non-collusion measures makes the system analysis more challenging and constitutes an interesting question for future research. 
We remark that the probability of collisions and overtaking shrinks as $\alpha$ and $\beta$ increase and as $\sigma$ decreases.
Our explicit analysis in Section~\ref{sec:OUprocess} in case of a quadratic potential allows for a quantification of these probabilities in the steady state distribution.
\end{remark}

\subsection{Port-Hamiltonian Formulation}
Next, we rewrite the system \eqref{eq:modn} in matrix form and identify a port-Hamiltonian structure.
\begin{propo}\label{prop:PHSform}
Denoting $Z(t)=(Q(t),p(t))^\top\in\R^{2N}$, $t\in[0,\infty)$, 
the dynamics of the periodic system \eqref{eq:modn} are given by
\begin{equation}\label{eq:PHS}
    \begin{aligned}
        dZ(t)=(J-R)\nabla H(Z(t))\,dt+G\, dW(t),&&&& Z(0)=z_0=(Q_0,p_0)^\top,
    \end{aligned}
\end{equation}
with 
\begin{equation*}
  J=\begin{bmatrix}0&\A\\-\A^\top &0 \end{bmatrix}\in \R^{2N\times 2N},\quad
  R=\begin{bmatrix}0&0\\0 &\beta \A^\top \A \end{bmatrix}\in \R^{2N\times 2N},\quad
  G=\begin{bmatrix}0\\\sigma I\end{bmatrix}\in \R^{2N\times N},
\end{equation*}
\begin{equation}\label{eq:def_A}
  \A=\begin{bmatrix}-1&1& & &\\
  &-1&1& &\\
  &&\ddots&\ddots&\\[1mm]
  &&&-1&1\\
  1&&&&-1\end{bmatrix}\in \R^{N\times N},
\end{equation}
and the Hamiltonian operator $H\colon \R^{2N} \to \R$,
\begin{equation}\label{eq:def_Hamiltonian}
    H(Q,p)=\frac{1}{2} \|p\|^2+\sum_{n=1}^N U(Q_n), \quad p\in \R^N, Q=(Q_1, \ldots, Q_n)^\top \in \R^N.
\end{equation}
Moreover, the matrix $J$ is skew-symmetric by $N\times N$ block while $R$ is symmetric positive semi-definite.
\end{propo}
\begin{proof}
First note that for all $(Q,p)^\top \in \R^{2N}$ we have
\begin{equation*}
 (\nabla H)(Q,p)=\begin{bmatrix}U'(Q)\\p\end{bmatrix},\qquad\text{with}\quad U'(Q)=\bigl(U'(Q_n)\bigr)_{n=1}^N.
\end{equation*}
Moreover, we have
\begin{equation*}
   \A^\top \A=\begin{bmatrix}2&-1&&&-1\\
  -1&2&-1&&\\
   &&\ddots&&\\[1mm]
   &&-1&2&-1\\
   -1&&&-1&2\end{bmatrix}\in\R^{N\times N}.
\end{equation*}
It follows directly from the model \eqref{eq:modn} that
\begin{equation}
    \begin{cases}
        dQ(t)&=\A p(t)\,dt,\\
        dp(t)&=\bigl(-\A^\top U'(Q(t))-\beta \A^\top \A p(t)\bigr)\,dt + \sigma\,dW(t)
    \end{cases}
\end{equation}
and hence
\begin{equation*}
    \begin{aligned}
        dZ(t)=\tilde B\nabla H(Z(t))\,dt + G\,dW(t)
    \end{aligned}
\end{equation*}
with
\begin{equation*}
\tilde B:=\begin{bmatrix}0&\A\\-\A^\top &-\beta \A^\top \A \end{bmatrix}=J-R.
\end{equation*}
Clearly, $J$ is skew-symmetric and $R$ is symmetric. Moreover, it holds for all $(Q,p)^\top \in \R^{2N}$ that 
\begin{equation*}
(Q^\top, p^{\top})R\begin{pmatrix}
    Q\\ p
\end{pmatrix}
=\beta p^{\top}\A^\top \A p=\beta \|\A p\|^2\in [0,\infty)
\end{equation*}
and hence $R$ is positive semidefinite. 
\end{proof}

\begin{remark}
In the port-Hamiltonian formulation of the model, both the kinematics between the distance and the relative velocity and the kinematics between the relative velocities and the distance potential terms are part of the Hamiltonian and the skew-symmetric matrix $J$ of the port-Hamiltonian system. 
These components represent the conservative part of the system. 

The velocity difference terms in the motion model correspond to the dissipation matrix $R$ in the port-Hamiltonian system, 
while the noise plays the role of an external input (some disturbances). 
The port-Hamiltonian system is linear in the sense that $J$, $R$ and $G$ are constant. 
The nonlinear components arise only from the distance-based potential $U$ and the Hamiltonian. 
\end{remark}

\subsection{Hamiltonian Behavior}
For $f\in C^2(\R^{2N},\R)$ and $Z(t)=(Q(t),p(t))^\top$, $t\in [0,\infty)$, given by \eqref{eq:PHS} It\^o's formula implies that
\begin{equation*}
    \begin{split}
        df(Z(t))&=\bigl(\nabla f(Z(t))\bigr)^\top 
        \bigl( (J-R)\nabla H(Z(t))\,dt + G\, dW(t)\bigr)
        +\frac{1}{2}\trace\bigl(G^\top (\nabla^2 f)(Z(t))G\bigr)\, dt \\
        &
        =\Bigl( \bigl(\nabla f(Z(t))\bigr)^\top (J-R)\nabla H(Z(t))
        +\frac{1}{2}\trace\bigl(G^\top (\nabla^2 f)(Z(t))G\bigr)\Bigr)\,dt 
        +\bigl(\nabla f(Z(t))\bigr)^\top G\, dW(t).
    \end{split}
\end{equation*}
Put differently, the generator $\mathcal L$ of $(Z(t))_{t\in [0,\infty)}$ satisfies for all sufficiently regular $f\colon \R^{2N}\to \R$ and $zc\in \R^{2N}$ that
\begin{equation*}
    (\mathcal{L}f)(z)=\bigl(\nabla f(z)\bigr)^{\top}(J-R)\nabla H(z) 
    +\frac{1}{2}\trace\bigl(G^\top (\nabla^2 f)(z)G\bigr).
\end{equation*}
Using the skew-symmetry of $J$ this implies for the Hamiltonian $H$ (given by \eqref{eq:def_Hamiltonian}) that
\begin{equation}\label{eq:dynHam}
    \begin{split}
        dH(Z(t))&
        =\Bigl( \bigl(\nabla H(Z(t))\bigr)^\top (J-R)\nabla H(Z(t)) +\frac{1}{2}\trace\bigl(G^\top (\nabla^2 H)(Z(t))G\bigr)\Bigr)\,dt \\
        &\qquad +\bigl(\nabla H(Z(t))\bigr)^\top G\, dW(t)\\
        &=\Bigl( -\bigl(\nabla H(Z(t))\bigr)^\top R\nabla H(Z(t)) +\frac{N\sigma^2 }{2}\Bigr)\,dt +\sigma p^\top (t) \,dW(t)\\
         &=\Bigl( -\beta \|\A p(t)\|^2+\frac{N\sigma^2 }{2}\Bigr)\,dt +\sigma p^{\top}(t) \, dW(t).\\
    \end{split}
\end{equation}
Note that the Hamiltonian behavior in time does not depend explicitly on the distance $Q$ and the potential $U$ thanks to the skew symmetry. 
Further remarks on the Hamiltonian behavior can be found below.
\begin{remark}\label{rem:Stab_det}
    The deterministic system ($\sigma=0$) is \textit{stable}, 
    i.e., the Hamiltonian is non-increasing over time. 
    Indeed, in the case $\sigma=0$ equation \eqref{eq:dynHam} reads for all $t\in [0,\infty)$
\begin{equation}\label{eq:dynHam_det}
    d H\bigl(Z(t)\bigr)=- \beta \|\A p(t)\|^2\, dt.
\end{equation}
Recall that according to Remark~\ref{rem:average_distance} and Remark~\ref{rem:average_velocity} the deterministic system \eqref{eq:modn} is always in a state where the ensemble's mean distance and the ensemble's mean velocity satisfy 
\begin{equation*}
  \frac{1}{N}\sum_{n=1}^N Q_n(t)=\frac{L}{N},\qquad
  \frac{1}{N}\sum_{n=1}^N p_n(t)=\frac{1}{N}\sum_{n=1}^N p_n(0).
\end{equation*}
By Jensen's inequality (using the convexity of $U$) it follows that the unique minimum of $H$ over all $(Q,p)^\top\in \R^{2N}$ with $\frac{1}{N}\sum_{n=1}^NQ_n=\frac{L}{N}$
and $\frac{1}{N}\sum_{n=1}^N p_n=\frac{1}{N}\sum_{n=1}^N p_n(0)$ 
is given by the uniform configuration $(Q^*,p^*)^\top\in\R^{2N}$ with $Q^*_n=\frac{L}{N}$
and $p^*_n=\frac{1}{N}\sum_{n=1}^Np_n(0)$ for all $n\in \{1,\ldots,N\}$.
Note that $(Q^*,p^*)^\top\in \R^{2N}$ is an equilibrium point of \eqref{eq:modn}. 

Since $\ker(\A)=\{\lambda p^*|\lambda \in \R\}$, we see in the case $\beta>0$ from \eqref{eq:dynHam_det} that $H(Z(t))$ is strictly decreasing whenever $p(t)\neq p^*$. 
Moreover, whenever $p(t)= p^*$ but $Q(t)\neq Q^*$ the dynamics~\eqref{eq:modn} 
ensure that $p$ is moved away from $p^*$ and thus $H$ is also decreasing in this situation. 
This indicates the convergence of the deterministic system to the uniform configuration $(Q^*,p^*)^\top$ as $t\to\infty$. 
We make this statement rigorous in the special case of a quadratic potential $U$ in Section~\ref{sec:OUprocess}.
\end{remark}
\begin{remark}
In contrast, the Hamiltonian for the stochastic system could increase in expectation over time. 
Indeed, if we start with uniform velocities $p_0(t)=p^*$, then it holds that 
\begin{equation*}
    \frac{d}{dt}\bigg|_{t=0}\E\bigl[H(Z(t))\bigr]=\frac{\sigma^2N}2>0,
\end{equation*}
    provided that $\sigma>0$ (here we tacitly assume sufficient regularity for the stochastic integral in \eqref{eq:dynHam} to vanish in expectation).
    For this reason, we cannot directly analyze the stability of the system using Lyapunov-style arguments in combination with the Hamiltonian.
    Indeed, in Remark~\ref{rem:divergence_p} we see that in general the stochastic system does not converge to a limiting distribution.
\end{remark}

\begin{remark}\label{rem:divergence_p}
Note that the ensemble's mean velocity $\xbar{p}$ satisfies 
\begin{equation*}
\xbar{p}(t)=\frac{1}{N}\sum_{n=1}^N p_n(t)
=\frac{1}{N}\mathbf{1}^\top p(t)
=\frac{1}{N}\mathbf{1}^\top
\begin{bmatrix}
    0 & I
\end{bmatrix}Z(t),\quad\text{for all}\quad t\ge0.
\end{equation*}
If $Z(t)$ converged weakly as $t\to \infty$ then by the continuous mapping theorem also $\xbar{p}(t)$ would converge weakly. 
However, by Remark~\ref{rem:average_velocity} the ensemble's mean velocity $\xbar{p}$ is a Brownian motion with variance $\sigma^2/N$ which does not converge weakly as $t\to\infty$ in the stochastic case $\sigma>0$. 
Hence $Z(t)$ does not converge weakly as $t\to\infty$ if $\sigma>0$. 
\end{remark}

\section{Explicit Collective Motion Analysis in case of a Quadratic Potential}\label{sec:OUprocess}
For $\alpha\in (0,\infty)$ we consider as the distance-based potential the quadratic function 
\begin{equation*}
    U(x)=\frac{(\alpha x)^2}{2},\qquad x\in\R.
\end{equation*}
In this case the system \eqref{eq:modn} is linear.
Indeed, the gradient of the Hamiltonian reads
\begin{equation*}
      (\nabla H)(z)=\begin{bmatrix}\alpha^2 Q\\p\end{bmatrix}, \qquad z=(Q, p)^\top,
\end{equation*}
and the system is an Ornstein-Uhlenbeck linear stochastic process
\begin{equation}\label{eq:OU}
    \begin{aligned}
        dZ(t)=\B Z(t)\,dt + G\,dW(t),&&&& Z(0)=z_0\in \R^{2N},
    \end{aligned}
\end{equation}
with 
\begin{equation*}
   \B=\begin{bmatrix}0&\A\\-\alpha^2\A^\top &-\beta \A^\top \A \end{bmatrix}
    \in \R^{2N\times 2N}\qquad\text{and}\qquad
     G=\begin{bmatrix}0\\\sigma I\end{bmatrix}\in\R^{2N\times N}
\end{equation*}
(recall the definition of $\A$ in \eqref{eq:def_A}).
The Hamiltonian for the quadratic potential is given by
\begin{equation}
    H(z)=\frac{1}{2} \|p\|^2+\frac{\alpha^2}{2}\|Q\|^2, \qquad z=(Q, p)^\top,
\end{equation}
The Ornstein-Uhlenbeck system \eqref{eq:OU} can be explicitly solved. 
We obtain using Duhamel's formula
\begin{equation*}
   Z(t)=e^{t\B}z(0)+\int_0^t e^{(t-s)\B}G\,dW(s), \qquad t\in [0,\infty),
\end{equation*}
(see, e.g., \cite[Section 4.4.6]{gardiner1985handbook} or \cite[Section 3.7]{pavliotis2014stochastic} for this and the further results on multivariate Ornstein-Uhlenbeck processes that we use in the sequel).
Furthermore, $(Z(t))_{t\in [0,\infty)}$ is a Gaussian process. 
In particular, for all $t\in [0,\infty)$ the random variable $Z(t)$ is normal with expectation
\begin{equation}\label{eq:mean_z}
    \mu_Z(t)=\E\bigl[Z(t)\bigr]=e^{t\B}z(0),
\end{equation}
and covariance matrix 
\begin{equation}\label{eq:cov_z}
    \Sigma_Z(t)=\E[(Z(t)-\mu_Z(t))(Z(t)-\mu_Z(t))^\top]\int_0^t e^{s\B}GG^\top e^{s\B^\top}\,ds.
\end{equation}

We aim at describing the system's limit behavior as $t\to\infty$. 
As we have seen in Remark~\ref{rem:divergence_p} in the stochastic case $\sigma>0$, 
the original system in $Z=(Q,p)$-coordinates does not converge weakly as $t\to\infty$. 
However, in Subsection~\ref{subsec:limit_X} we show that passing from velocity coordinates $p$ to deviations $D$ from the ensemble's mean velocity leads to stable dynamics. 
We formally introduce the process $D$ in the next subsection.

\subsection{Deviation from Ensemble Mean Velocity}\label{subsec:deviation}
In the sequel we analyze the agents' deviation from the ensemble's mean velocity.
To do so, recall that
\begin{equation*}
  \xbar{p}(t)=\frac{1}{N}\sum_{n=1}^N p_n(t), \qquad t\in [0,\infty),
\end{equation*}
is the ensemble's mean velocity and recall from Remark~\ref{rem:average_velocity} that $\xbar{p}$
is a Brownian motion with variance $\sigma^2/N$.
We introduce the deviation of agent $n$ from the ensemble's mean velocity as 
\begin{equation*}
   D_n(t)=p_n(t)-\xbar{p}(t)
   =\Bigl(1-\frac{1}{N}\Bigr) p_n(t) - \frac{1}{N}\sum_{k\neq n}p_k(t),\qquad t\ge 0.
\end{equation*}
Let
\begin{equation}\label{eq:def_M} 
M=\frac{1}{N}
\begin{bmatrix}
N-1&-1&-1&\ldots &-1\\
  -1&N-1&-1&\ldots&-1\\
   &&\ddots&&\\[1mm]
  -1 & \ldots &-1&N-1&-1\\
   -1&\ldots& -1 &-1&N-1\end{bmatrix}\in\R^{N\times N}.
\end{equation}
Then, the deviation vector $D(t)=(D_n(t))_{n=1}^N$ at time $t\ge 0$ is given by 
\begin{equation}\label{eq:def_D}
  D(t)=Mp(t).
\end{equation}
Next, we introduce the new processes $X(t)=(Q(t), D(t))^\top\in \R^{2N}$, $t\ge 0$. 
Note that 
\begin{equation}\label{eq:def_X}
X(t)=\begin{bmatrix}
    Q(t)\\
    Mp(t)
\end{bmatrix}
=\begin{bmatrix}
    I & 0\\
    0 & M
\end{bmatrix}Z(t), \qquad t\ge 0.
\end{equation}
Since for every $t\in [0,\infty)$ the random variable $Z(t)$ is normally distributed with expectation $\mu_Z(t)$ (given by \eqref{eq:mean_z}) and covariance matrix $\Sigma_Z(t)$ (given by \eqref{eq:cov_z}) we obtain that
for every $t\in [0,\infty)$ the random variable $X(t)$ is normally distributed with expectation 
\begin{equation}\label{eq:mean_X}
\mu_X(t)=
\begin{bmatrix}
    I & 0\\
    0 & M
\end{bmatrix}\mu_Z(t)
=\begin{bmatrix}
    I & 0\\
    0 & M
\end{bmatrix}e^{t\B}z(0).
\end{equation}
and covariance matrix 
\begin{equation}\label{eq:cov_X}
\Sigma_X(t)=
\begin{bmatrix}
    I & 0\\
    0 & M
\end{bmatrix}\Sigma_Z(t)
\begin{bmatrix}
    I & 0\\
    0 & M^\top
\end{bmatrix}
=\int_0^t
\begin{bmatrix}
    I & 0\\
    0 & M
\end{bmatrix} e^{s\B}GG^\top e^{s\B^\top}
\begin{bmatrix}
    I & 0\\
    0 & M
\end{bmatrix}
\,ds.
\end{equation}
Moreover, note that $\A M=\A$ and $M\A=\A$. 
Hence $\A^\top M=\A^\top$ and $M\A^\top=\A^\top$. 
This implies
$$dD(t)=\left(-\alpha^2 \A^\top Q(t)-\beta \A^\top \A D(t)\right)dt+\sigma M\,dW(t)$$ 
and hence $(X(t))_{t\ge 0}$ satisfies the dynamics
\begin{equation}\label{eq:Xprocess}
dX(t)=BX(t)\,dt+\begin{bmatrix}
    0\\
    \sigma M
\end{bmatrix}\,dW(t).
\end{equation}

We aim at explicitly describing the limit distribution of $(X(t))_{t\ge 0}$ as $t\to \infty$ 
(see Theorem~\ref{thm:limitX} for the main result in this regard). 
To this end, we follow the instructive route to first compute $\mu_X(t)$ and $\Sigma_X(t)$ (and also $\mu_Z(t)$ and $\Sigma_Z(t)$) for fixed $t\ge 0$ explicitly and then determine the limits $\mu_X(\infty)$ and $\Sigma_X(\infty)$ as $t\to\infty$ (see also Remark~\ref{rem:lyapunov_equation} for a verification of $\Sigma_X(\infty)$ via the Lyapunov equation associated to \eqref{eq:Xprocess}). 
To compute the matrix exponentials $e^{tB}$, $t\ge 0$, showing up in \eqref{eq:mean_X} and \eqref{eq:cov_X} we first provide an eigendecomposition of the matrix $\B$.

\subsection{Eigenanalysis for the Matrices $\A$ and $\B$}\label{subsec:eigenAB}
\newcommand*{\W}{\mathcal{W}}

In this subsection we provide the complex eigendecomposition $\W \Lambda \W^{-1}$ of $\B$ which allows to compute the matrix exponentials $e^{tB}$, $t\ge0$, that characterize the expectation vectors $\mu_Z(t)$ and covariance matrices $\Sigma_Z(t)$. 
Note that the matrix $B$ is not symmetric and not even normal. Using the eigendecomposition of the circulant matrix $A$ we show that $B$ still admits an eigendecomposition if for all $j\in \{1,2,\dots,N-1\}$ we have $\beta^2\mu_j-4\alpha^2\neq 0$. 
To this end, we first fix some notation that we use throughout the section.

\begin{setting}\label{set:eigenana}.
Let $\omega=e^{\frac{2\pi\iu}{N}}\in \C$. 
For all $j\in \{0,1,\ldots,N-1\}$ let $\kappa_j=\omega^j-1\in \C$ and
$$
  v_j=\frac{1}{\sqrt{N}}(1,\omega^j,\omega^{2j},\ldots, \omega^{(N-1)j})^\top\in \C^N,$$
and let $\mu_j=2-2\cos\left(\frac{2\pi j}{N} \right)\in [0,4]$. 
For all $j\in \{0,1,\ldots, N-1\}$, $k\in \{1,2\}$ 
let
\begin{equation}\label{eq:eigenvalue}
\lambda_{j,k}=\frac{1}{2}(-\beta\mu_j+(-1)^k\sqrt{\beta^2\mu^2_j-4\alpha^2 \mu_j})\in \C.
\end{equation}
For all
$j\in \{1,2,\ldots, N-1\}$, $k\in \{1,2\}$ let
\begin{equation*}
  w_{0,1}=\begin{bmatrix} I \\ 0 \end{bmatrix} v_0\in \C^{2N}, \quad 
  w_{0,2}=\begin{bmatrix} 0 \\ I \end{bmatrix} v_0\in \C^{2N}, \quad 
  w_{j,k}=\begin{bmatrix} \kappa_j I\\ \lambda_{j,k} I \end{bmatrix} v_j\in \C^{2N}.
\end{equation*}
By $\W\in \C^{2N\times 2N}$ we denote the matrix whose columns are given by the vectors $w_{l,k}$, $l\in \{0,1,\ldots, N-1\}$, $k\in \{1,2\}$, i.e., 
    \begin{equation*}
    \W=\begin{bmatrix}
        w_{0,1} & w_{1,1} & \ldots & w_{N-1,1} & w_{0,2} & w_{1,2} & \ldots w_{N-1,2}
    \end{bmatrix}\in \C^{2N}.
    \end{equation*}
    If for all $j\in \{1,2,\dots,N-1\}$ it holds that $\beta^2\mu_j-4\alpha^2\neq 0$, then for all $j\in \{1,2,\ldots, N-1\}$, $k\in \{1,2\}$ we introduce
 \begin{equation*}
  u_{0,1}=\begin{bmatrix} I \\ 0 \end{bmatrix} v_0\in \C^{2N}, \quad 
  u_{0,2}=\begin{bmatrix} 0 \\ I \end{bmatrix} v_0\in \C^{2N}, \quad 
  u_{j,k}=\begin{bmatrix} (-1)^k\frac{\bar{\lambda}_{j,3-k}}{\bar{\kappa}_j (\bar{\lambda}_{j,1}-\bar{\lambda}_{j,2})} I\\ 
  (-1)^k\frac{1}{\bar{\lambda}_{j,2}-\bar{\lambda}_{j,1}} I 
  \end{bmatrix} v_j\in \C^{2N}.
\end{equation*}
Finally, we introduce the matrix
\begin{equation}\label{eq:Klm1}
    K=\sum_{j=1}^{N-1}\frac{v_jv_j^*}{\mu_j}\in \C^{N\times N}.
\end{equation}
\end{setting}

The first result of this subsection provides $N$ vectors that are eigenvectors of both $\A$ and $\A^\top$. 
Moreover, it presents the corresponding eigenvalues. 
Since these eigenvalues are distinct, it follows that the $N$ eigenvectors are linearly independent. 
\begin{lemma}\label{lem:evA}
The family $(v_j)_{j\in \{0,1,\ldots,N-1\}}$ is an orthonormal basis of $\C^N$ 
(with respect to the standard inner product $\langle u, v\rangle=u^* v=\bar{u}^\top v$ on $\C^N$) and for every $j\in \{0,1,\ldots,N-1\}$ it holds that
$v_j$ is an eigenvector of $\A$ with eigenvalue $\kappa_j=\omega^j-1$ and $v_j$ is an eigenvector of $\A^\top$ with eigenvalue $\bar{\kappa}_j=\omega^{(N-1)j}-1$. 
Moreover, it holds for all $j\in \{0,1,\ldots,N-1\}$ that 
\begin{equation*}
  \kappa_j \bar{\kappa}_j
  =2\Bigl(1-\cos\Bigl(\frac{2\pi j}{N}\Bigr)\Bigr)
  =\mu_j=4\sin^2\Bigl(\frac{\pi j}{N}\Bigr).
\end{equation*}
\end{lemma}
\begin{proof}
   For the first part we refer the reader to \cite[Chapter~4]{Fong07}.
    Next, note that we have the following circular property of the unit roots $\omega$
    \begin{equation*}
   \omega^{(N-1)j}=\exp\Bigl(\frac{2\pi\iu}{N}(N-1)j\Bigr)
    =\exp\Bigl(2\pi\iu j- \frac{2\pi\iu }{N} j\Bigr)=
    \exp\Bigl(-\frac{2\pi\iu }{N} j\Bigr)=\omega^{-j}.
    \end{equation*}
    This implies that
    \begin{equation*}
    \kappa_j \bar{\kappa}_j =(\omega^j-1)(\omega^{-j}-1)=2-(\omega^j+\omega^{-j})
    =2-2\cos\Bigl(\frac{2\pi j}{N}\Bigr).
    \end{equation*}
    \begin{equation*}
    =2-2\Big[\cos^2\Bigl(\frac{\pi j}{N}\Bigr)-\sin^2\Bigl(\frac{\pi j}{N}\Bigr)\Bigr]
    = 4\sin^2\Bigl(\frac{\pi j}{N}\Bigr).
    \end{equation*}
\end{proof}

The next lemma shows that from each eigenvector $v$ of $\A$ that is also an eigenvector of $\A^\top$ we can create (up to) two eigenvectors of $\B$. 

\begin{lemma}\label{lem:fromAtoB}
    Let $v\in \C^N$, $\kappa,\tilde \kappa \in \C$ and 
    assume that $v$ is an eigenvector of $\A$ and $\A^\top$ with eigenvalues $\kappa$ and $\tilde \kappa$, respectively, i.e., $\A v=\kappa v$ and $\A^\top v=\tilde \kappa v$. Let $\lambda_1,\lambda_2\in \C$ be the complex roots of $z\mapsto z^2+\beta \kappa \tilde \kappa z+\alpha^2 \kappa \tilde \kappa$ and for $j\in \{1,2\}$ let
    \begin{equation*}
      w_j=\begin{bmatrix} \kappa I\\ \lambda_j I \end{bmatrix}v\in \C^{2N}.
    \end{equation*}
    Then for all $j\in \{1,2\}$ we have that $\B w_j=\lambda_j w_j$.
\end{lemma}
\begin{note}
    In the case $\kappa=0$ we have $w_j=0$ for $j\in \{1,2\}$ and hence $w_j$ is not an eigenvector of $B$.
\end{note}

\begin{proof}
For all $j\in \{1,2\}$ it holds that
\begin{equation*}
    \begin{split}
        \B w_j&=\begin{bmatrix}0&\A\\-\alpha^2 \A^\top &-\beta \A^\top \A \end{bmatrix}
        \begin{bmatrix} \kappa I\\ \lambda_j I \end{bmatrix} v
= \begin{bmatrix} \lambda_j \A\\ -\alpha^2 \kappa\A^\top-\beta \lambda_j \A^\top \A \end{bmatrix} v
= \begin{bmatrix} \kappa \lambda_j v \\ -\alpha^2 \kappa \tilde \kappa v-\beta \kappa \tilde \kappa \lambda_j v \end{bmatrix}\\
&
= \begin{bmatrix} \kappa \lambda_j v \\ \lambda_j^2 v \end{bmatrix}
=\lambda_j \begin{bmatrix} \kappa I\\ \lambda_j I \end{bmatrix} v = \lambda_j w_j.
    \end{split}
\end{equation*}
\end{proof}

The next result shows that under the condition that $\alpha\neq 0$ and that for all $j\in \{1,2,\ldots, N-1\}$ it holds that $\beta^2(1-\cos\left(\frac{2\pi j}{N}\right))\neq 2\alpha^2$ that $\B$ is diagonalizable and provides explicit representations of its eigenvalues and eigenvectors.
\begin{propo}\label{prop:Bdiagonal}
Assume that for all $j\in \{1,2,\dots,N-1\}$ it holds that $\beta^2\mu_j-4\alpha^2\neq 0$.
Then $w_{j,k}$, $j\in \{0,1,\dots,N-1\}$, $k\in \{1,2\}$, is a basis of $\C^{2N}$ consisting of eigenvectors of $\B$ with eigenvalues $\lambda_{j,k}$, $j\in \{0,1,\dots,N-1\}$, $k\in \{1,2\}$. 
Moreover, for all $j\in \{0,1,\dots, \lfloor \frac{N}{2}\rfloor\}$, $l\in \{1,2,\dots, \lfloor \frac{N}{2}\rfloor\}$, $k,m\in \{1,2\}$ with $(j,k)\neq (l,m)$ it holds that $\lambda_{j,k}\neq \lambda_{l,m}$ and for all $j\in \{1,2,\dots, \lfloor \frac{N}{2}\rfloor\}$, $k\in \{1,2\}$ it holds that $\lambda_{j,k}=\lambda_{N-j,k}$. 
In particular, $\B$ has exactly $1+2\lfloor \frac{N}{2}\rfloor$ different eigenvalues.
\end{propo}

\begin{note}
Before proceeding with the proof of Proposition~\ref{prop:Bdiagonal}, 
we remark that if there exists 
$j\in \{1,2,\ldots, N-1\}$ such that $\beta^2(1-\cos\left(\frac{2\pi j}{N}\right))= 2\alpha^2$ then $\lambda_{j,1}=\lambda_{j,2}$. 
In this case the geometric multiplicity of $\lambda_{j,1}=\lambda_{j,2}$ is one and hence smaller than the algebraic multiplicity (two). 
Thus $\B$ is not diagonalizable in this situation. 
We refer to Remark~\ref{rem:rid_cond} and Remark~\ref{rem:rid_cond2} on how to dispense with this condition in the computation of the expectations and covariance matrices of $X$ and $Z$.
\end{note}

\begin{remark}\label{rem:disc_ev} Let us briefly discuss the form of the eigenvalues $\lambda_{j,k}$ of $\B$ and the consequences on the long-term behavior of $Z$ and $X$. We assume again that
 $\beta^2(1-\cos\left(\frac{2\pi j}{N}\right))\neq  2\alpha^2$ for all $j\in \{1,2,\ldots, N-1\}$.

First, note that $\lambda_{0,1}=\lambda_{0,2}=0$. 
For $j\in \{1,\ldots,N-1\}$, $k\in \{1,2\}$, the eigenvalue $\lambda_{j,k}$ is real if and only if $\beta^2(1-\cos\left(\frac{2\pi j}{N}\right))\ge 2\alpha^2$. In this case $\lambda_{j,k}$ is strictly negative, since we assume throughout that $\alpha, \beta>0$.
For $j\in \{1,\ldots,N-1\}$, $k\in \{1,2\}$, the eigenvalue $\lambda_{j,k}$ has non-vanishing imaginary part if and only if $\beta^2(1-\cos\left(\frac{2\pi j}{N}\right))< 2\alpha^2$.
In this case $\lambda_{j,1}=\bar{\lambda}_{j,2}$ and the real part of $\lambda_{j,k}$ equals $-\frac{1}{2}\beta \mu_j$ which is strictly negative.

To sum up, we see that under our assumption $\alpha, \beta>0$ we have two zero eigenvalues and $2N-2$ eigenvalues with strictly negative real parts. 
As we will see in Proposition~\ref{prop:expl_cov} the two zero eigenvalues $\lambda_{0,1}$ and $\lambda_{0,2}$ lead to the divergence of $\Sigma_Z(t)$ as $t\to \infty$ (cf.\ Remark~\ref{rem:divergence_p}). 
For the process $X$, however, we will see that in the direction of the associated eigenvectors $w_{0,1}$ and $w_{0,2}$ there is also no noise component. This together with the negativity of the real parts of the remaining eigenvalues ensures convergence of $X$.
\end{remark}

\begin{note}
    All eigenvalues of $\B$ are real if and only if $\beta^2(1-\cos(2\pi/N))\ge 2\alpha^2$. 
    In this case, the dynamics are overdamped and do not describe any oscillations. 
    In contrast, all eigenvalues $\lambda_{j,k}$, $j\in \{1,\ldots,N-1\}$, $k\in \{1,2\}$, 
    have non-vanishing imaginary part if and only if $\beta^2(1-\cos(2\pi\lfloor N/2 \rfloor /N))<2\alpha^2$. 
    A sufficient condition (which is also necessary if $N$ is even) is $\beta<\alpha$. 
    In this case, the dynamics are underdamped and the system oscillates for every angular frequency $\theta_j=2\pi j/N$, $j\in \{1,\ldots,N-1\}$. 
    More generally, all eigenvalues $\lambda_{j,k}$ such that $\beta^2(1-\cos(2\pi j/N))<2\alpha^2$ have non-vanishing imaginary parts. 
    If such eigenvalues exist, then the system is underdamped and oscillates at angular frequencies $\theta_j=2\pi j/N$ for which $\beta^2(1-\cos(\theta_j))<2\alpha^2$. 
    The frequency $\theta_{j_0}$ is critically damped (neither underdamped nor overdamped) when $\beta^2(1-\cos(\theta_{j_0})=2\alpha^2$.
\end{note}

\begin{proof}
    We first show that for all $j\in \{0,1,\ldots, N-1\}$, $k\in \{1,2\}$, the vector $w_{j,k}$ is an eigenvector of $\B$. 
    For $j=0$, $k\in \{1,2\}$ it follows from that fact that $\A$ and $\A^\top$ have zero row sums that $\B w_{0,k}=0$ and hence that $w_{0,k}$ is an eigenvector with eigenvalue $\lambda_{0,k}=0$ of $\B$. 
    For $j\in \{1,2,\ldots, N-1\}$, $k\in \{1,2\}$ Lemma~\ref{lem:evA} shows that $v_j$ is an eigenvector of $\A$ with eigenvalue $\kappa_j$ and that $v_j$ is an eigenvector of $\A^\top$ with eigenvalue $\tilde \kappa_j=\omega^{(N-1)j}-1$. 
    Moreover, it holds that $\kappa_j \tilde \kappa_j=\mu_j$. 
    Note that $\lambda_{j,k}$, $k\in \{1,2\}$, are the complex roots of $z\mapsto z^2+\beta \mu_j z +\alpha^2 \mu_j$. 
    Hence, Lemma \ref{lem:fromAtoB} implies that $w_{j,k}$ is an eigenvector of $\B$ for all $j\in \{1,2,\ldots, N-1\}$, $k\in \{1,2\}$ (note that $\kappa_j\neq 0$ for all $j\in \{1,2,\ldots, N-1\}$ and hence $w_{j,k}\neq 0$). 

It remains to be shown that $w_{j,k}$, $j\in \{0,1,\ldots, N-1\}$, $k\in \{1,2\}$, are linearly independent. 
To this end note that for $j\in \{1,2, \ldots, N-1\}$ we have the symmetry property 
\begin{equation*}
     \mu_j=2-2\cos\left(\frac{2\pi j}{N} \right)=2-2\cos\left(2\pi-\frac{2\pi j}{N} \right)=2-2\cos\left(\frac{2\pi (N-j)}{N} \right)=\mu_{N-j}.
\end{equation*}
and hence $\lambda_{j,k}=\lambda_{N-j,k}$ for all $j\in \{0,1, \ldots, N-1\}$, $k\in \{1,2\}$. This implies that
\begin{equation}\label{eq:set_eval}
    \bigl\{\lambda_{j,k}\colon j\in \{0,1,\dots, N-1\}, k\in \{1,2\}\bigr\}
    =\{0\}\cup \bigl\{\lambda_{j,k}\colon j\in \left\{1,2\dots, \lfloor N/2\rfloor \right\}, k\in \{1,2\}\bigr\}.
\end{equation}
We next show that the right-hand side of \eqref{eq:set_eval} consists of distinct elements. 
Indeed, since $\alpha\neq 0$, we have that $0\notin \left\{\lambda_{j,k}\colon j\in \left\{1,2\ldots, \lfloor N/2\rfloor \right\}, k\in \{1,2\}\right\}$. 
Moreover, note that for all $j,l\in \{1,2,\ldots, \lfloor \frac{N}{2}\rfloor\}$ with $j\ne l$ it holds that $\mu_j\neq \mu_l$. 
This implies that for all $j,l\in \{1,2,\ldots, \lfloor \frac{N}{2}\rfloor\}$ with $j\ne l$ and $k,m\in \{1,2\}$ it holds that $\lambda_{j,k}\neq \lambda_{l,m}$. 
Finally, for all $j\in \{1,2,\ldots, \lfloor \frac{N}{2}\rfloor\}$ it holds that $\lambda_{j,1}\neq \lambda_{j,2}$ since $\mu_j\neq 0$ and $\beta^2\mu_j-4\alpha^2\neq0$.
This shows that the right-hand side of \eqref{eq:set_eval} consists of distinct elements.
Since each element in the right-hand side of \eqref{eq:set_eval} is an eigenvalue of $\B$ it follows that the associated eigenspaces form a direct sum in $\C^{2N}$. 

To conclude the proof we show that the dimension of this direct sum is $2N$. 
Since $w_{0,1}$ and $w_{0,2}$ are linearly independent the eigenspace associated to $0$ has at least dimension $2$. 
Next, for every $j\in \{1,2,\dots, \lfloor \frac{N-1}{2}\rfloor\}$, $k\in \{1,2\}$, there are the eigenvectors $w_{j,k}$ and $w_{N-j,k}$ associated to the eigenvalues $\lambda_{j,k}$. 
Since $v_j$ and $v_{N-j}$ are eigenvectors of $\A$ with distinct eigenvalues $\kappa_j$ and $\kappa_{N-j}$, it follows that $w_{j,k}$ and $w_{N-j,k}$ are linearly independent. 
This implies that for every $j\in \{1,2,\dots, \lfloor \frac{N-1}{2}\rfloor\}$, $k\in \{1,2\}$ the eigenspace associated to $\lambda_{j,k}$ has at least dimension $2$.

If $N$ is odd, then $\lfloor \frac{N-1}{2}\rfloor=\lfloor \frac{N}{2}\rfloor$ and the sum of the dimensions of all these eigenspaces is at least $2+2\cdot 2 \cdot \lfloor \frac{N-1}{2}\rfloor=2N$ and hence the proof is complete.
If $N$ is even then $\lfloor \frac{N}{2}\rfloor=\frac{N}{2}>\lfloor \frac{N-1}{2}\rfloor$ and we additionally have the eigenspaces associated to $\lambda_{N/2,k}$ for $k\in \{1,2\}$, 
which have at least dimension $1$ (each of them contains the vector $w_{N/2,k}$, respectively).
Hence, in the case, where $N$ is even the sum of the dimensions of all these eigenspaces is at least $2+2\cdot 2 \cdot \lfloor \frac{N-1}{2}\rfloor+2=2N$. 
This completes the proof.
\end{proof}

Proposition \ref{prop:Bdiagonal} provides the matrix of eigenvectors $\W$ of $B$. 
Lemma~\ref{lem:Winv} presents its inverse $\W^{-1}$.

\begin{lemma}\label{lem:Winv}
Assume that for all $j\in \{1,2,\dots,N-1\}$ it holds that $\beta^2\mu_j-4\alpha^2\neq 0$.
Then $\W^{-1}$ is the complex conjugate of the matrix whose columns are given by the vectors $u_{j,k}$, $j\in \{0,1,\ldots, N-1\}$, $k\in \{1,2\}$, i.e., 
    \begin{equation}\label{eq:Winv}
        \W^{-1}=\begin{bmatrix}
        u_{0,1} & u_{1,1} & \ldots & u_{N-1,1} & u_{0,2} & u_{1,2} & \ldots u_{N-1,2}
    \end{bmatrix}^*\in \C^{2N}.
    \end{equation}
\end{lemma}
\begin{proof}
    First note that due to the assumption that for all $j\in \{1,2,\ldots, N-1\}$ we have $\beta^2\mu_j-4\alpha^2\neq 0$, the vectors $u_{j,k}$, $j\in \{0,1,\ldots, N-1\}$, $k\in \{1,2\}$, are well-defined (in the sense that we do not divide by zero). 

    Next, denote by $\mathcal{U}$ the right-hand side of \eqref{eq:Winv}. We first consider the first $N$ rows of the product $\mathcal{U} \W$. Note that
    \begin{equation*}
(\mathcal U \W)_{1,1}=u^*_{0,1}w_{0,1}=v_0^*v_0=1.
    \end{equation*}
Furthermore, note that 
\begin{equation*}
(\mathcal U \W)_{1,N}=u^*_{0,1}w_{0,2}=0.
    \end{equation*}
    The fact that $v_0$ is orthogonal to all $v_j$, $j\in \{1,\ldots,N-1\}$, implies for all $m\in \{1,\ldots,2N\} \setminus \{1,N\}$ that
    \begin{equation*}
(\mathcal U \W)_{1,m}=0.
  \end{equation*}
    The fact that $v_j$, $j\in \{0,\ldots, N-1\}$, forms an orthonormal basis of $\C^N$ implies for all $l,m\in \{2,\ldots, N-1\}$
    \begin{equation*}
(\mathcal U \W)_{l,m}
= u^*_{l-1,1}w_{m-1,1}
=\left(-\frac{ \kappa_{m-1}\lambda_{l-1,2}}{ \kappa_{l-1} ( \lambda_{l-1,1}- \lambda_{l-1,2})}
-\frac{\lambda_{m-1,1}}{ \lambda_{l-1,2}- \lambda_{l-1,1}}
\right)v^*_{l-1}v_{m-1}=\delta_{l,m}.
\end{equation*}
Moreover, for all $l\in \{2,\ldots, N-1\}$, $m\in \{N+1,\ldots, 2N\}$ we have
\begin{equation*}
(\mathcal U \W)_{l,m}
= u^*_{l-1,1}w_{m-1-N,2}
=\left(-\frac{ \kappa_{m-1-N}\lambda_{l-1,2}}{ \kappa_{l-1} ( \lambda_{l-1,1}- \lambda_{l-1,2})}
-\frac{\lambda_{m-1-N,2}}{ \lambda_{l-1,2}- \lambda_{l-1,1}}
\right)v^*_{l-1}v_{m-1-N}=0.
\end{equation*}
Thus the claim for the first $N$ rows of $\mathcal U \W$ is established. The claim for the last $N$ rows follows
in a similar way.
\end{proof}

\subsection{Explicit Representations of the Time-Marginal Distributions of $Z$ and $X$}

In this subsection we use the eigendecomposition of $\B$ established in Subsection \ref{subsec:eigenAB} to obtain for every $t\ge 0$ explicit representations of the expectation vectors $\mu_X(t)$ and $\mu_Z(t)$ and the covariance matrices $\Sigma_X(t)$ and $\Sigma_Z(t)$.
Throughout this subsection we use the notation of Setting \ref{set:eigenana}.

Let $\Lambda \in \C^{2N\times 2N}$ be the diagonal matrix with the vector 
\begin{equation*}
\begin{bmatrix}
    \lambda_{0,1} & \lambda_{1,1} & \ldots & \lambda_{N-1,1} & 
    \lambda_{0,2} & \lambda_{1,2} & \ldots & \lambda_{N-1,2} 
\end{bmatrix} \in \C^{2N}
\end{equation*}
on its diagonal. Then we have that 
\begin{equation*}
   B=\W\Lambda \W^{-1}.
\end{equation*}
This implies for all $t\ge 0$ that
\begin{equation*}
   e^{tB}=\W e^{t\Lambda}\W^{-1}.
\end{equation*}
With Lemma \ref{lem:Winv} we obtain
\begin{equation}\label{eq:exptB}
    \begin{split}
        e^{tB}&=\sum_{j=0}^{N-1}\sum_{k=1}^2e^{t\lambda_{j,k}}w_{j,k}u^*_{j,k}
=\begin{bmatrix}
    v_0v_0^* & 0\\
    0 &  v_0v_0^*
\end{bmatrix}
+
\sum_{j=1}^{N-1}\sum_{k=1}^2e^{t\lambda_{j,k}}
\begin{bmatrix}
   \frac{ (-1)^k\lambda_{j,3-k}}{ ( \lambda_{j,1}- \lambda_{j,2})} v_jv_j^* & \frac{(-1)^k\kappa_j}{ \lambda_{j,2}- \lambda_{j,1}} v_jv_j^* \\
   \frac{(-1)^k \lambda_{j,k}\lambda_{j,3-k}}{ \kappa_j ( \lambda_{j,1}- \lambda_{j,2})}v_j v_j^* & 
   \frac{(-1)^k\lambda_{j,k}}{ \lambda_{j,2}- \lambda_{j,1}} v_j v_j^*
\end{bmatrix}\\
&=
\begin{bmatrix}
    v_0v_0^* & 0\\
    0 &  v_0v_0^*
\end{bmatrix}
+
\sum_{j=1}^{N-1}
\begin{bmatrix}
   \frac{ e^{t\lambda_{j,2}}\lambda_{j,1}-e^{t\lambda_{j,1}}\lambda_{j,2}}{  \lambda_{j,1}- \lambda_{j,2}} v_jv_j^* & \frac{(e^{t\lambda_{j,2}}-e^{t\lambda_{j,1}})\kappa_j}{ \lambda_{j,2}- \lambda_{j,1}} v_jv_j^* \\
   \frac{(e^{t\lambda_{j,2}}- e^{t\lambda_{j,1}})\lambda_{j,1}\lambda_{j,2}}{ \kappa_j ( \lambda_{j,1}- \lambda_{j,2})}v_j v_j^* & 
   \frac{e^{t\lambda_{j,2}}\lambda_{j,2}-e^{t\lambda_{j,1}}\lambda_{j,1}}{ \lambda_{j,2}- \lambda_{j,1}} v_j v_j^*
\end{bmatrix}.
    \end{split}
\end{equation}
Moreover, note that the facts that $v_0$ is an eigenvector of $M$ with eigenvalue $0$ and that for every $j\in \{1,\ldots,N-1\}$ the vector $v_j$ is an eigenvector of $M$ with eigenvalue $1$ imply that
\begin{equation}\label{eq:IMetB}
\begin{bmatrix}
    I & 0\\
    0 & M
\end{bmatrix}e^{t\B}
=
\begin{bmatrix}
    v_0v_0^* & 0\\
    0 &  0
\end{bmatrix}
+
\sum_{j=1}^{N-1}
\begin{bmatrix}
   \frac{ e^{t\lambda_{j,2}}\lambda_{j,1}-e^{t\lambda_{j,1}}\lambda_{j,2}}{  \lambda_{j,1}- \lambda_{j,2}} v_jv_j^* & \frac{(e^{t\lambda_{j,2}}-e^{t\lambda_{j,1}})\kappa_j}{ \lambda_{j,2}- \lambda_{j,1}} v_jv_j^* \\
   \frac{(e^{t\lambda_{j,2}}- e^{t\lambda_{j,1}})\lambda_{j,1}\lambda_{j,2}}{ \kappa_j ( \lambda_{j,1}- \lambda_{j,2})}v_j v_j^* & 
   \frac{e^{t\lambda_{j,2}}\lambda_{j,2}-e^{t\lambda_{j,1}}\lambda_{j,1}}{ \lambda_{j,2}- \lambda_{j,1}} v_j v_j^*
\end{bmatrix}.
\end{equation}
These representations lead to the following result about the expectations $\mu_Z(t)$ and $\mu_X(t)$.

\begin{lemma}\label{lem:mean_t_expl} 
Assume that for all $j\in \{1,2,\ldots, N-1\}$ we have $\beta^2(1-\cos\left(\frac{2\pi j}{N}\right))\neq 2\alpha^2$. Then for every
 $t\ge 0$ we have
    \begin{equation}\label{eq:mean_Z_t_expl}
        \mu_Z(t)=\begin{bmatrix}
    \frac{L}{N}\mathbf{1}\\
    \xbar{p}(0) \mathbf{1}
\end{bmatrix}
+
\sum_{j=1}^{N-1}
\begin{bmatrix}
   \frac{(v_j^*Q(0)) (e^{t\lambda_{j,2}}\lambda_{j,1}-e^{t\lambda_{j,1}}\lambda_{j,2})-(v_j^*p(0))(e^{t\lambda_{j,2}}-e^{t\lambda_{j,1}})\kappa_j}{  \lambda_{j,1}- \lambda_{j,2}} v_j \\
   \frac{(v_j^*Q(0))(e^{t\lambda_{j,2}}- e^{t\lambda_{j,1}})\lambda_{j,1}\lambda_{j,2}-(v_j^*p(0))(e^{t\lambda_{j,2}}\lambda_{j,2}-e^{t\lambda_{j,1}}\lambda_{j,1})\kappa_j}{ \kappa_j ( \lambda_{j,1}- \lambda_{j,2})}v_j 
\end{bmatrix}
    \end{equation}
and 
\begin{equation}\label{eq:mean_X_t_expl}
    \mu_X(t)=\begin{bmatrix}
    \frac{L}{N}\mathbf{1}\\
    0
\end{bmatrix}
+
\sum_{j=1}^{N-1}
\begin{bmatrix}
   \frac{(v_j^*Q(0)) (e^{t\lambda_{j,2}}\lambda_{j,1}-e^{t\lambda_{j,1}}\lambda_{j,2})-(v_j^*p(0))(e^{t\lambda_{j,2}}-e^{t\lambda_{j,1}})\kappa_j}{  \lambda_{j,1}- \lambda_{j,2}} v_j \\
   \frac{(v_j^*Q(0))(e^{t\lambda_{j,2}}- e^{t\lambda_{j,1}})\lambda_{j,1}\lambda_{j,2}-(v_j^*p(0))(e^{t\lambda_{j,2}}\lambda_{j,2}-e^{t\lambda_{j,1}}\lambda_{j,1})\kappa_j}{ \kappa_j ( \lambda_{j,1}- \lambda_{j,2})}v_j 
\end{bmatrix},
\end{equation}
where $\xbar{p}(0)=\frac{1}{N}\sum_{n=1}^Np_n(0)$.
\end{lemma}

\begin{proof}
    Combining \eqref{eq:mean_z} and \eqref{eq:exptB} yields for every $t\ge 0$ that
    \begin{equation*}
    \mu_Z(t)=e^{t\B}Z(0)=\begin{bmatrix}
    (v_0^*Q(0))v_0\\
    (v_0^*p(0))v_0
\end{bmatrix}
+
\sum_{j=1}^{N-1}
\begin{bmatrix}
   \frac{ (v_j^*Q(0))(e^{t\lambda_{j,2}}\lambda_{j,1}-e^{t\lambda_{j,1}}\lambda_{j,2})}{  \lambda_{j,1}- \lambda_{j,2}} v_j + \frac{(v_j^*p(0))(e^{t\lambda_{j,2}}-e^{t\lambda_{j,1}})\kappa_j}{ \lambda_{j,2}- \lambda_{j,1}} v_j \\
   \frac{(v_j^*Q(0))(e^{t\lambda_{j,2}}- e^{t\lambda_{j,1}})\lambda_{j,1}\lambda_{j,2}}{ \kappa_j ( \lambda_{j,1}- \lambda_{j,2})}v_j  + 
   \frac{(v_j^*p(0))(e^{t\lambda_{j,2}}\lambda_{j,2}-e^{t\lambda_{j,1}}\lambda_{j,1})}{ \lambda_{j,2}- \lambda_{j,1}} v_j 
\end{bmatrix}.
    \end{equation*}
Since $v_0=\frac{1}{\sqrt{N}}\mathbf{1}$, $\sum_{k=1}^NQ_{k}(0)=L$, and $(v_0^*p_0)v_0=\xbar{p}_0 \mathbf{1}$ we obtain \eqref{eq:mean_Z_t_expl}. Similarly, using \eqref{eq:mean_X} and \eqref{eq:IMetB}  yields \eqref{eq:mean_X_t_expl}.
\end{proof}

\begin{remark}\label{rem:rid_cond}
Since the proof of Lemma~\ref{lem:mean_t_expl} relies on the eigendecomposition of $\B$ from Subsection~\ref{subsec:eigenAB} we also have to impose here the condition that 
for all $j\in \{1,2,\ldots, N-1\}$ we have $\beta^2(1-\cos\left(\frac{2\pi j}{N}\right))\neq 2\alpha^2$. 
This assumption in particular ensures that for all $j\in \{1,2,\ldots, N-1\}$ it holds $\lambda_{j,1}\neq\lambda_{j,2}$ and hence that the expressions in \eqref{eq:mean_Z_t_expl} and \eqref{eq:mean_X_t_expl} are well-defined. 
But with Lemma~\ref{lem:mean_t_expl} it is also possible to obtain representations of $\mu_X(t)$ and $\mu_Z(t)$ in the case where there exists $j\in \{1,2,\ldots, N-1\}$ with $\beta^2(1-\cos\left(\frac{2\pi j}{N}\right))= 2\alpha^2$. 
Indeed, in this case there exists a sequence $(\alpha_k)_{k\in \N}\subseteq (0,\infty)$ such that $\lim_{k\to \infty}\alpha_k=\alpha$ and for all $j\in \{1,2,\ldots, N-1\}$, $k\in \N$ we have $\beta^2(1-\cos\left(\frac{2\pi j}{N}\right))\neq 2\alpha^2_k$. 

It follows from classical stability results of SDEs that for every $t\ge 0$ the random variables $Z_k(t)$ and $X_k(t)$ converge to $Z(t)$ and $X(t)$, respectively (where $(Z_k(t))_{t\ge 0}$ and $(X_k(t))_{t\ge 0}$ are given by \eqref{eq:OU} and \eqref{eq:def_X} with $\alpha$ replaced by $\alpha_k$, respectively). 
In particular, it follows for every $t\ge 0$ that $\mu_{Z_k}(t)\to \mu_Z(t)$ and  $\mu_{X_k}(t)\to \mu_X(t)$ as $k\to \infty$. 
The explicit representations of $\mu_Z(t)$ and $\mu_X(t)$ are thus obtained by taking the limit $\lambda_{j,1}-\lambda_{j,2}\to 0$ in \eqref{eq:mean_Z_t_expl} and \eqref{eq:mean_X_t_expl}.
\end{remark}

We next turn to the computation of the covariance matrices $\Sigma_Z(t)$ and $\Sigma_X(t)$.

\begin{propo}\label{prop:expl_cov}
    Assume that for all $j\in \{1,2,\ldots, N-1\}$ we have $\beta^2(1-\cos\left(\frac{2\pi j}{N}\right))\neq 2\alpha^2$. Then for every $t\ge 0$ we have
    \begin{equation}\label{eq:cov_z_exp}
\begin{split}
    \Sigma_Z(t)
&=
\begin{bmatrix}
     \frac{\sigma^2}{2\alpha^2 \beta} K &
     0
     \\ 
0
     & 
    \sigma^2 t v_0v_0^*+ \frac{\sigma^2}{2 \beta} K
\end{bmatrix}
+
\sigma^2
\sum_{j=1}^{N-1}
\begin{bmatrix}
     a_j(t) v_jv_j^* &
     \bar{b}_j(t) v_jv_j^*
     \\ 
b_j(t)v_jv_j^*
     & 
     c_j(t) v_j v_j^*
\end{bmatrix}
\end{split}
\end{equation}
and
\begin{equation}\label{eq:cov_x_exp}
\begin{split}
    \Sigma_X(t)
&=
\begin{bmatrix}
     \frac{\sigma^2}{2\alpha^2 \beta} K &
     0
     \\ 
0
     & 
    \frac{\sigma^2}{2 \beta} K
\end{bmatrix}
+
\sigma^2
\sum_{j=1}^{N-1}
\begin{bmatrix}
     a_j(t) v_jv_j^* &
     \bar{b}_j(t) v_jv_j^*
     \\ 
b_j(t)v_jv_j^*
     & 
     c_j(t) v_j v_j^*
\end{bmatrix},
\end{split}
\end{equation}
where
\begin{equation*}
a_j(t)=\frac{
\beta\lambda_{j,1}e^{2t\lambda_{j,2}}
+4\alpha^2 e^{-t\beta \mu_j}
+\beta\lambda_{j,2}e^{2t\lambda_{j,1}}
}{2\alpha^2\beta \mu_j (\beta^2\mu_j-4\alpha^2) },\quad c_j(t)=\frac{\lambda_{j,2}\beta e^{2t\lambda_{j,2}}+4\alpha^2 e^{-t\beta\mu_j}+\lambda_{j,1}\beta e^{2t\lambda_{j,1}}}{ 2\beta(\beta^2\mu_j^2-4\alpha^2 \mu_j)
},
\end{equation*}
and
\begin{equation*}
b_j(t)=\frac{\left[
\beta \mu_j(e^{2t\lambda_{j,2}}+e^{2t\lambda_{j,1}})
+2(\lambda_{j,2}+\lambda_{j,1})e^{-t\beta \mu_j}
\right]
\bar{\kappa}_j}{2\beta \mu_j(\beta^2\mu_j^2-4\alpha^2 \mu_j)}.
\end{equation*}

\end{propo}
\begin{proof}
    By \eqref{eq:exptB} we have for every $t\ge 0$ that
    \begin{equation*}
    \begin{split}
       e^{tB}G&=
\begin{bmatrix}
    v_0v_0^* & 0\\
    0 &  v_0v_0^*
\end{bmatrix}
\begin{bmatrix}
    0\\
    \sigma I
\end{bmatrix}
+
\sum_{j=1}^{N-1}
\begin{bmatrix}
   \frac{ e^{t\lambda_{j,2}}\lambda_{j,1}-e^{t\lambda_{j,1}}\lambda_{j,2}}{  \lambda_{j,1}- \lambda_{j,2}} v_jv_j^* & \frac{(e^{t\lambda_{j,2}}-e^{t\lambda_{j,1}})\kappa_j}{ \lambda_{j,2}- \lambda_{j,1}} v_jv_j^* \\
   \frac{(e^{t\lambda_{j,2}}- e^{t\lambda_{j,1}})\lambda_{j,1}\lambda_{j,2}}{ \kappa_j ( \lambda_{j,1}- \lambda_{j,2})}v_j v_j^* & 
   \frac{e^{t\lambda_{j,2}}\lambda_{j,2}-e^{t\lambda_{j,1}}\lambda_{j,1}}{ \lambda_{j,2}- \lambda_{j,1}} v_j v_j^*
\end{bmatrix}
\begin{bmatrix}
    0\\
    \sigma I
\end{bmatrix}\\
&= \sigma\left(
\begin{bmatrix}
    0\\
    v_0v_0^*
\end{bmatrix}
+
\sum_{j=1}^{N-1}
\begin{bmatrix}
     \frac{(e^{t\lambda_{j,2}}-e^{t\lambda_{j,1}})\kappa_j}{ \lambda_{j,2}- \lambda_{j,1}} v_jv_j^* \\ 
   \frac{e^{t\lambda_{j,2}}\lambda_{j,2}-e^{t\lambda_{j,1}}\lambda_{j,1}}{ \lambda_{j,2}- \lambda_{j,1}} v_j v_j^*
\end{bmatrix}
\right)
\end{split}
    \end{equation*}
This and the fact that $\B$ and $G$ are real-valued imply for every $t\ge 0$ that
\begin{equation*}
    \begin{split}
       G^\top e^{tB^\top}&=
       \left(e^{tB}G\right)^\top=
       \left(e^{tB}G\right)^*\\
&= \sigma\left(
\begin{bmatrix}
    0 &
    v_0v_0^*
\end{bmatrix}
+
\sum_{j=1}^{N-1}
\begin{bmatrix}
     \frac{(e^{t \bar{\lambda}_{j,2}}-e^{t\bar{\lambda}_{j,1}})\bar{\kappa}_j}{ \bar{\lambda}_{j,2}- \bar{\lambda}_{j,1}} v_jv_j^* & 
   \frac{e^{t\bar{\lambda}_{j,2}} \bar{\lambda}_{j,2}-e^{t\bar{\lambda}_{j,1}}\bar{\lambda}_{j,1}}{ \bar{\lambda}_{j,2}- \bar{\lambda}_{j,1}} v_j v_j^*
\end{bmatrix}
\right)
\end{split}
    \end{equation*}
Combining these two equations yields for every $t\ge 0$ that
\begin{multline*}
    e^{tB}GG^\top e^{tB^\top}
    =
    \sigma^2\left(
\begin{bmatrix}
    0\\
    v_0v_0^*
\end{bmatrix}
+
\sum_{j=1}^{N-1}
\begin{bmatrix}
     \frac{(e^{t \lambda_{j,2}}-e^{t \lambda_{j,1}})\kappa_j}{ \lambda_{j,2}- \lambda_{j,1}} v_jv_j^* \\ 
   \frac{e^{t \lambda_{j,2}}\lambda_{j,2}-e^{t\lambda_{j,1}}\lambda_{j,1}}{ \lambda_{j,2}- \lambda_{j,1}} v_j v_j^*
\end{bmatrix}
\right)\\
\cdot
\left(
\begin{bmatrix}
    0 &
    v_0v_0^*
\end{bmatrix}
+
\sum_{j=1}^{N-1}
\begin{bmatrix}
     \frac{(e^{t\bar{\lambda}_{j,2}}-e^{t\bar{\lambda}_{j,1}})\bar{\kappa}_j}{ \bar{\lambda}_{j,2}- \bar{\lambda}_{j,1}} v_jv_j^* & 
\frac{e^{t\bar{\lambda}_{j,2}}\bar{\lambda}_{j,2}-e^{t\bar{\lambda}_{j,1}}\bar{\lambda}_{j,1}}{ \bar{\lambda}_{j,2}- \bar{\lambda}_{j,1}} v_j v_j^*
\end{bmatrix}
\right).
\end{multline*}
The fact that $v_j$, $j\in \{0,1,\ldots, N-1\}$, forms an orthonormal basis of $\C^N$ further implies that
\begin{equation}\label{eq:eGGe}
    e^{tB}GG^\top e^{tB^\top}
    =
    \sigma^2\left(
\begin{bmatrix}
    0 & 0\\
    0& v_0v_0^*
\end{bmatrix}
+
\sum_{j=1}^{N-1}
\begin{bmatrix}
     \mathfrak{a}_j(t) v_jv_j^* &
     \bar{\mathfrak{b}}_j(t) v_jv_j^*
     \\ 
\mathfrak{b}_j(t)v_jv_j^*
     & 
     \mathfrak{c}_j(t) v_j v_j^*
\end{bmatrix}
\right),
\end{equation}
where
$$
\mathfrak{a}_j(t)=\left|\frac{(e^{t\lambda_{j,2}}-e^{t\lambda_{j,1}})\kappa_j}{ \lambda_{j,2}- \lambda_{j,1}}\right|^2, \quad \mathfrak{b}_j(t)=\frac{(e^{t\lambda_{j,2}}\lambda_{j,2}-e^{t\lambda_{j,1}}\lambda_{j,1})(e^{t\bar{\lambda}_{j,2}}-e^{t\bar{\lambda}_{j,1}})\bar{\kappa}_j}{ |\lambda_{j,2}- \lambda_{j,1}|^2}
$$
and
$$
\mathfrak{c}_j(t)=\left|\frac{e^{t\lambda_{j,2}}\lambda_{j,2}-e^{t\lambda_{j,1}}\lambda_{j,1}}{ \lambda_{j,2}- \lambda_{j,1}}\right|^2.
$$
Note that for every $j\in \{1,\ldots, N-1\}$ we have
$$
|\lambda_{j,2}-\lambda_{j,1}|^2=|\beta^2\mu_j^2-4\alpha^2 \mu_j| \quad \text{and} \quad |\kappa_j|^2=\mu_j.
$$
This implies that
$$
\mathfrak{a}_j(t)=\frac{|e^{t\lambda_{j,2}}-e^{t\lambda_{j,1}}|^2\mu_j}{ |\beta^2\mu_j^2-4\alpha^2 \mu_j|}=
\frac{e^{2t\lambda_{j,2}}-2e^{-t\beta \mu_j}+e^{2t\lambda_{j,1}}}{ \beta^2\mu_j-4\alpha^2 },
$$
$$
\mathfrak{b}_j(t)=\frac{(e^{t\lambda_{j,2}}\lambda_{j,2}-e^{t\lambda_{j,1}}\lambda_{j,1})(e^{t\bar{\lambda}_{j,2}}-e^{t\bar{\lambda}_{j,1}})\bar{\kappa}_j}{|\beta^2\mu_j^2-4\alpha^2 \mu_j|}
=\frac{\left[
\lambda_{j,2}(e^{2t\lambda_{j,2}}-e^{-t\beta \mu_j})+
\lambda_{j,1}(e^{2t\lambda_{j,1}}-e^{-t\beta \mu_j})\right]
\bar{\kappa}_j}{\beta^2\mu_j^2-4\alpha^2 \mu_j}
$$
and 
$$
\mathfrak{c}_j(t)=\frac{|e^{t\lambda_{j,2}}\lambda_{j,2}-e^{t\lambda_{j,1}}\lambda_{j,1}|^2}{ |\beta^2\mu_j^2-4\alpha^2 \mu_j|}
=\frac{\lambda_{j,2}^2e^{2t\lambda_{j,2}}-2\alpha^2\mu_je^{-t\beta\mu_j}+\lambda_{j,1}^2e^{2t\lambda_{j,1}}}{ \beta^2\mu_j^2-4\alpha^2 \mu_j}.
$$
Therefore we have for all $t\ge 0$, $j\in \{1,\ldots, N-1\}$
\begin{equation*}
    \begin{split}
        \int_0^t\mathfrak{a}_j(s)ds&=
\frac{
\frac{1}{2\lambda_{j,2}}(e^{2t\lambda_{j,2}}-1)
+\frac{2}{\beta \mu_j}(e^{-t\beta \mu_j}-1)
+\frac{1}{2\lambda_{j,1}}(e^{2t\lambda_{j,1}}-1)}{ \beta^2\mu_j-4\alpha^2 }\\
&=
\frac{
\beta\mu_j\lambda_{j,1}(e^{2t\lambda_{j,2}}-1)
+4\lambda_{j,1}\lambda_{j,2}(e^{-t\beta \mu_j}-1)
+\beta\mu_j\lambda_{j,2}(e^{2t\lambda_{j,1}}-1)}{2\lambda_{j,2}\lambda_{j,1}\beta\mu_j (\beta^2\mu_j-4\alpha^2) }\\
&=
\frac{
\beta\lambda_{j,1}e^{2t\lambda_{j,2}}
+4\alpha^2 e^{-t\beta \mu_j}
+\beta\lambda_{j,2}e^{2t\lambda_{j,1}}
-(\beta\lambda_{j,1}+4\alpha^2 +\beta\lambda_{j,2})
}{2\alpha^2\beta \mu_j (\beta^2\mu_j-4\alpha^2) }\\
&=
\frac{1}{2\alpha^2\beta \mu_j}+
\frac{
\beta\lambda_{j,1}e^{2t\lambda_{j,2}}
+4\alpha^2 e^{-t\beta \mu_j}
+\beta\lambda_{j,2}e^{2t\lambda_{j,1}}
}{2\alpha^2\beta \mu_j (\beta^2\mu_j-4\alpha^2) }
=\frac{1}{2\alpha^2\beta \mu_j}+a_j(t)
.
    \end{split}
\end{equation*}
Moreover, we have for all $t\ge 0$, $j\in \{1,\ldots, N-1\}$
\begin{equation*}
    \begin{split}
\int_0^t \mathfrak{b}_j(s)ds&=
\frac{\left[
\frac{1}{2}(e^{2t\lambda_{j,2}}-1)
+\frac{\lambda_{j,2}}{\beta \mu_j}(e^{-t\beta \mu_j}-1)
+\frac{1}{2}(e^{2t\lambda_{j,1}}-1)
+\frac{\lambda_{j,1}}{\beta \mu_j}(e^{-t\beta \mu_j}-1)
\right]
\bar{\kappa}_j}{\beta^2\mu_j^2-4\alpha^2 \mu_j}\\
&=
\frac{\left[
\beta \mu_j(e^{2t\lambda_{j,2}}-1)
+2\lambda_{j,2}(e^{-t\beta \mu_j}-1)
+\beta \mu_j(e^{2t\lambda_{j,1}}-1)
+2\lambda_{j,1}(e^{-t\beta \mu_j}-1)
\right]
\bar{\kappa}_j}{2\beta \mu_j(\beta^2\mu_j^2-4\alpha^2 \mu_j)}
\\
&=
\frac{\left[
\beta \mu_j(e^{2t\lambda_{j,2}}+e^{2t\lambda_{j,1}})
+2(\lambda_{j,2}+\lambda_{j,1})e^{-t\beta \mu_j}
\right]
\bar{\kappa}_j}{2\beta \mu_j(\beta^2\mu_j^2-4\alpha^2 \mu_j)}=b_j(t).
    \end{split}
\end{equation*}
Finally, we have for all $t\ge 0$, $j\in \{1,\ldots, N-1\}$
\begin{equation*}
    \begin{split}
\int_0^t \mathfrak{c}_j(s)ds&=\frac{\frac{\lambda_{j,2}}{2}(e^{2t\lambda_{j,2}}-1)+\frac{2\alpha^2}{\beta}(e^{-t\beta\mu_j}-1)+\frac{\lambda_{j,1}}{2}(e^{2t\lambda_{j,1}}-1)}{ \beta^2\mu_j^2-4\alpha^2 \mu_j}\\
&=\frac{\lambda_{j,2}\beta(e^{2t\lambda_{j,2}}-1)+4\alpha^2(e^{-t\beta\mu_j}-1)+\lambda_{j,1}\beta(e^{2t\lambda_{j,1}}-1)}{ 2\beta(\beta^2\mu_j^2-4\alpha^2 \mu_j)}
\\
&=\frac{\lambda_{j,2}\beta e^{2t\lambda_{j,2}}+4\alpha^2 e^{-t\beta\mu_j}+\lambda_{j,1}\beta e^{2t\lambda_{j,1}}-(\beta \lambda_{j,2}+4\alpha^2 +\beta \lambda_{j,1})}{ 2\beta(\beta^2\mu_j^2-4\alpha^2 \mu_j)
}
\\
&=
\frac{1}{2\beta \mu_j}+
\frac{\lambda_{j,2}\beta e^{2t\lambda_{j,2}}+4\alpha^2 e^{-t\beta\mu_j}+\lambda_{j,1}\beta e^{2t\lambda_{j,1}}}{ 2\beta(\beta^2\mu_j^2-4\alpha^2 \mu_j)
}
=\frac{1}{2\beta \mu_j}+c_j(t).
    \end{split}
\end{equation*}
Combining \eqref{eq:cov_z} and \eqref{eq:eGGe} with these equations  yields for every $t\ge 0$ that
\begin{equation}\label{eq:cov_z_exp0}
\begin{split}
    \Sigma_Z(t)&=\int_0^t e^{s\B}GG^\top e^{s\B^\top}\,ds
    = 
    \int_0^t
    \sigma^2\left(
\begin{bmatrix}
    0 & 0\\
    0& v_0v_0^*
\end{bmatrix}
+
\sum_{j=1}^{N-1}
\begin{bmatrix}
     \mathfrak{a}_j(s) v_jv_j^* &
     \overline{\mathfrak{b}}_j(s) v_jv_j^*
     \\ 
\mathfrak{b}_j(s)v_jv_j^*
     & 
     \mathfrak{c}_j(s) v_j v_j^*
\end{bmatrix}
\right)\,ds\\
&=
\sigma^2\left(
\begin{bmatrix}
    0 & 0\\
    0& v_0v_0^*t
\end{bmatrix}
+
\begin{bmatrix}
     \frac{1}{2\alpha^2 \beta} \sum_{j=1}^{N-1}\frac{v_jv_j^*}{\mu_j} &
     0
     \\ 
0
     & 
     \frac{1}{2 \beta} \sum_{j=1}^{N-1}\frac{v_jv_j^*}{\mu_j}
\end{bmatrix}
+
\sum_{j=1}^{N-1}
\begin{bmatrix}
     a_j(t) v_jv_j^* &
     \bar{b}_j(t) v_jv_j^*
     \\ 
b_j(t)v_jv_j^*
     & 
     c_j(t) v_j v_j^*
\end{bmatrix}
\right).
\end{split}
\end{equation}
Moreover, by \eqref{eq:cov_X} and the facts that $Mv_0=0$ and $Mv_j=v_j$, $j\in \{1,\ldots,N-1\}$, we obtain for every $t\ge 0$
\begin{equation}\label{eq:cov_x_exp0}
\begin{split}
\Sigma_X(t)&=
\begin{bmatrix}
    I & 0\\
    0 & M
\end{bmatrix}\Sigma_Z(t)
\begin{bmatrix}
    I & 0\\
    0 & M^\top
\end{bmatrix}
\\&=
\sigma^2\left(
\begin{bmatrix}
     \frac{1}{2\alpha^2 \beta} \sum_{j=1}^{N-1}\frac{v_jv_j^*}{\mu_j} &
     0
     \\ 
0
     & 
     \frac{1}{2 \beta} \sum_{j=1}^{N-1}\frac{v_jv_j^*}{\mu_j}
\end{bmatrix}
+
\sum_{j=1}^{N-1}
\begin{bmatrix}
     a_j(t) v_jv_j^* &
     \bar{b}_j(t) v_jv_j^*
     \\ 
b_j(t)v_jv_j^*
     & 
     c_j(t) v_j v_j^*
\end{bmatrix}
\right).
\end{split}
\end{equation}
Recalling the definition \eqref{eq:Klm1} of $K$ \eqref{eq:cov_z_exp0} and \eqref{eq:cov_x_exp0} yield \eqref{eq:cov_z_exp} and \eqref{eq:cov_x_exp}, respectively.
\end{proof}
\begin{remark}\label{rem:rid_cond2}
    Similarly to Remark \ref{rem:rid_cond}, by taking limits Proposition \ref{prop:expl_cov} also provides explicit representations of $\Sigma_Z(t)$ and $\Sigma_X(t)$ in the case where there exists $j\in \{1,2,\ldots, N-1\}$ with $\beta^2(1-\cos\left(\frac{2\pi j}{N}\right))= 2\alpha^2$. 
\end{remark}

\subsection{Limit Distribution of $X$}\label{subsec:limit_X}
In this subsection we take the limits in Lemma \ref{lem:mean_t_expl} and Proposition \ref{prop:expl_cov} to obtain an explicit representation of the limit distribution of $(X(t))_{t\ge 0}$ as $t\to \infty$. To this end, recall the definition of the matrix $
K=\sum_{j=1}^{N-1}\frac{v_jv_j^*}{\mu_j}\in \C^{\N\times \N}$ in Setting \ref{set:eigenana}.

\begin{theo}\label{thm:limitX}
The process $(X(t))_{t\ge 0}$ given by \eqref{eq:Xprocess} converges for $t\to\infty$ in distribution to a normal distribution with expectation 
$$
\mu_X(\infty)=
\begin{bmatrix}
\frac{L}{N}\mathbf{1}\\
    0
\end{bmatrix} \in \R^{2N}
$$
and covariance matrix
$$
\Sigma_X(\infty)=
\begin{bmatrix}
    \frac{\sigma^2}{2\alpha^2 \beta }K & 0\\
    0 & \frac{\sigma^2}{2 \beta }K
\end{bmatrix}\in \R^{2N\times 2N}.
$$
\end{theo}
\begin{proof}
For all $j\in \{1,\ldots,N\}$, $k\in \{1,2\}$, we have that $\Real(\lambda_{j,k})<0$ (see also Remark \ref{rem:disc_ev}). Under the assumption that for all $j\in \{1,2,\ldots, N-1\}$ it holds that $\beta^2\bigl(1-\cos\bigl(\frac{2\pi j}{N}\bigr)\bigr)\neq 2\alpha^2$ this together with Lemma~\ref{lem:mean_t_expl} and Proposition \ref{prop:expl_cov} yields that $\lim_{t \to \infty}\mu_X(t)=\mu_X(\infty)$ and $\lim_{t \to \infty}\Sigma_X(t)=\Sigma_X(\infty)$. As outlined in Remark \ref{rem:rid_cond} and Remark \ref{rem:rid_cond2} this convergence also holds in the case where the assumption is not satisfied.

Since for every $t\ge 0$ the random variable $X(t)$ is normal with expectation vector $\mu_X(t)$ and covariance matrix $\Sigma_X(t)$ it follows with Lévy's continuity theorem that $(X(t))_{t\ge 0}$ converges weakly to a normal distribution with expectation $\mu_X(\infty)$ and covariance matrix $\Sigma_X(\infty)$ as $t\to \infty$.
\end{proof}

\begin{remark}
    By Lemma \ref{lem:mean_t_expl} also the expectation vectors $\mu_Z(t)$  of $Z(t)$ converge to 
    $$
    \mu_Z(\infty)=\begin{bmatrix}
        \frac{L}{N}\mathbf{1}\\
        \xbar{p}(0)\mathbf{1}
    \end{bmatrix}$$
    as $t\to \infty$.
    However, the covariance matrices $\Sigma_Z(t)$, $t\ge 0$, do not converge in the case $\sigma \neq 0$ and hence $(Z(t))_{t\ge 0}$ cannot converge as we have already observed in Remark \ref{rem:divergence_p}.
\end{remark}

\begin{remark}
    Note that the matrix $K$ is singular. Indeed, the fact that $v_0=\mathbf{1}$ is orthogonal to all $v_j$, $j\in \{1,\ldots,N-1\}$, implies that $Kv_0=0$. The reason for this is, as already outlined in Remark \ref{rem:average_distance}, that our system is overdetermined. For example, we have for all $t\ge 0$ that $Q_N(t)=L-\sum_{n=1}^{N-1}Q_n(t)$ and $D_N(t)=-\sum_{n=1}^{N-1}D_n(t)$. It follows that the limit distribution is degenerate. By eliminating for example the last row and last column of $K$ one obtains a non-degenerate normal distribution and can reconstruct the last components as outlined above.
\end{remark}

\begin{remark}\label{rem:lyapunov_equation}
The steady-state covariance matrix $\Sigma_X(\infty)$ is in general also characterized by the matrix Lyapunov equation associated to the system \eqref{eq:Xprocess} (see, e.g., \cite[Section 4.4.6]{gardiner1985handbook} or \cite[Section 3.7]{pavliotis2014stochastic}). In our setting this equation reads
$$
\B\Sigma_X(\infty)+\Sigma_X(\infty)\B^\top=-\sigma^2\begin{bmatrix}
    0\\M
\end{bmatrix}
\begin{bmatrix}
    0 & M^\top
\end{bmatrix}.
$$
Using the representation of $\Sigma_X(\infty)$ from Theorem \ref{thm:limitX} this is equivalent to
$$
\begin{bmatrix}0&\A\\-\alpha^2\A^\top &-\beta \A^\top \A \end{bmatrix}
\begin{bmatrix}
    \frac{\sigma^2}{2\alpha^2 \beta}K & 0\\
    0 & \frac{\sigma^2}{2 \beta}K
\end{bmatrix}+
\begin{bmatrix}
    \frac{\sigma^2}{2\alpha^2 \beta}K & 0\\
    0 & \frac{\sigma^2}{2 \beta}K
\end{bmatrix}
\begin{bmatrix}0&-\alpha^2\A\\ \A^\top &-\beta \A^\top \A \end{bmatrix}
=-\sigma^2\begin{bmatrix}
    0\\M
\end{bmatrix}
\begin{bmatrix}
    0 & M
\end{bmatrix}
$$
and hence to
$$
\begin{bmatrix}
0 & \frac{\sigma^2}{2\beta}(\A K-K \A)\\
\frac{\sigma^2}{2\beta}(K \A^\top-\A^\top K)& -\frac{\sigma^2}{2}(\A^\top \A K+K \A^\top \A)
\end{bmatrix}
=-\sigma^2\begin{bmatrix}
    0 & 0\\ 0 & M^2
\end{bmatrix}.
$$
Next, note that by Lemma \ref{lem:evA} we have for all $j\in \{1,\ldots,N-1\}$ that $Av_j=\kappa_jv_j$ and $A^\top v_j=\bar \kappa_j v_j$. This implies
$$
\A K=\sum_{j=1}^{N-1}\frac{(Av_j)v_j^*}{\mu_j}=\sum_{j=1}^{N-1}\frac{(\kappa_j v_j)v_j^*}{\mu_j}=\sum_{j=1}^{N-1}\frac{v_j(\bar \kappa_j v_j)^*}{\mu_j}=\sum_{j=1}^{N-1}\frac{v_j(A^\top v_j)^*}{\mu_j}=K \A.
$$
Similarly, we obtain $K \A^\top=\A^\top K$. Again by Lemma \ref{lem:evA} we have for all $j\in \{1,\ldots,N-1\}$ that $\kappa_j \bar \kappa_j=\mu_j$. This implies that
\begin{equation*}
   \begin{split}
       \A^\top \A K+K \A^\top \A&=
\sum_{j=1}^{N-1}\frac{A^\top (Av_j)v_j^*}{\mu_j}+\frac{v_j(Av_j)^*A}{\mu_j}=
\sum_{j=1}^{N-1}\frac{\kappa_j (A^\top v_j)v_j^*}{\mu_j}+\frac{\bar \kappa_j v_j(A^\top v_j)^*}{\mu_j}\\
&=2\sum_{j=1}^{N-1}v_j v_j^*=2M.
   \end{split} 
\end{equation*}
Hence, we see that $\Sigma_X(\infty)$ indeed satisfies the matrix Lyapunov equation. 

\end{remark}

Before discussing Theorem~\ref{thm:limitX} in detail, we derive
in the next result a closed form representation of the matrix $K$. 
To do so, we realize that the entries of $K$ are instances of so-called \textit{Dowker’s sums in twisted form}
(the diagonal elements of $K$ have the untwisted form) \cite{cvijovic2007closed}.
This family of cosecant sums 
having closed forms that involve higher-order Bernoulli polynomials
was first presented
in a series of papers by Dowker 
describing his theory 
 of the Casimir effect \cite{dowker1987casimir}, asymptotic
expansion of the integrated heat kernel on cones \cite{dowker1989heat}, and in connection 
with the celebrated Verlinde’s formula for the dimensions of vector bundles on moduli spaces \cite{dowker1992verlinde}.

\begin{propo}\label{prop:exp_K}
For all $l,m\in \{1,\ldots,N\}$ the $(l,m)$-entry of $K$ satisfies
\begin{equation}\label{eq:K_final}
    K_{l,m}=\frac{1}{2N}
    \biggl[
    (l-m)^2-N|l-m|+ \frac{N^2-1}{6}\biggr].
\end{equation}
\end{propo}
\begin{proof}
    First note that for $j\in \{1,\dots, N-1\}$ and $l,m\in\{1,\dots, N\}$ the $(l,m)$-entry of the matrix $v_jv_j^*$ is given by
\begin{equation*}
   (v_jv_j^*)_{l,m}=v_{j,l}\bar{v}_{j,m}=\frac{1}{N}\omega^{(l-1)j}\overline{\omega^{(m-1)j}}=\frac{1}{N}e^{\frac{2\pi\iu j (l-m)}{N}}.
\end{equation*}
This implies that for all $l,m\in\{1,\ldots, N\}$ the $(l,m)$-entry $K$ satisfies
\begin{equation}\label{eq:Klm}
   K_{l,m}=\frac{1}{N}\sum_{j=1}^{N-1}\frac{1}{\mu_j}e^{\frac{2\pi\iu j (l-m)}{N}}
   =\frac{1}{N}\sum_{j=1}^{N-1}\frac{\cos\left(\frac{2\pi j (l-m)}{N}\right) + \iu\sin \left(\frac{2\pi j (l-m)}{N}\right)}{2-2\cos\left(\frac{2\pi j}{N} \right)}.
\end{equation}
The symmetries of $\sin$ and $\cos$ then imply that
\begin{equation}\label{eq:Klm2}
    K_{l,m}=\frac{1}{2N}\sum_{j=1}^{N-1}\frac{\cos\Bigl(\frac{2\pi j (l-m)}{N}\Bigr)}{1-\cos\Bigl(\frac{2\pi j}{N} \Bigr)}, \qquad l,m \in \{1,\dots,N\}. 
\end{equation}
From this we see that $K$ is symmetric and therefore we fix w.l.o.g.\ $l,m \in \{1,\dots,N\}$ with $l\ge m$ in the sequel.
Using that for all $j\in \{1,\ldots,N-1\}$ we have
$$
2-2\cos\Bigl(\frac{2\pi j}{N} \Bigr)=4\sin^2\Bigl(\frac{\pi j}{N} \Bigr)=\frac{4}{\csc^2\Bigl(\frac{\pi j}{N} \Bigr)}
$$ 
we rewrite \eqref{eq:Klm2} in the form
\begin{equation}\label{eq:Klm3}
    K_{l,m}=\frac{1}{4N}\sum_{j=1}^{N-1}
    \cos\Bigl(\frac{2\pi j (l-m)}{N}\Bigr)
    \csc^2\Bigl(\frac{\pi j}{N} \Bigr).
\end{equation}
This is a special case of \textit{Dowker's general family of cosecant sums} \cite[Section~4]{da2018generalized}
\begin{equation*}
   C_{2n}(N,r):=\sum_{j=1}^{N-1}
  \cos\Bigl(\frac{2\pi j r}{N}\Bigr)
   \csc^{2n}\Bigl(\frac{\pi j}{N} \Bigr),
\end{equation*}
where $n$ is a positive integer, $N$ is any integer except 1, and $r$ is an integer between 0 and $N-1$.
Here, we have to consider $C_2(N,l-m)$.
Using
\cite[Eq.~(2.3)]{Cvijovic2012} we can rewrite \eqref{eq:Klm3} as
\begin{equation}\label{eq:Klm4}
\begin{split}
    K_{l,m}&=\frac{2}{4N}
    \sum_{k=0}^1\binom{2}{2k}B_{2k}\Bigl(\frac{l-m}{N}\Bigr)
    B_{2-2k}^{(2)}(1)N^{2k} \\
    &=\frac{1}{2N}
    \biggl[B_0\Bigl(\frac{l-m}{N}\Bigr)
    B_2^{(2)}(1)+B_2\Bigl(\frac{l-m}{N}\Bigr)
    B_0^{(2)}(1)N^2\biggr].
    \end{split}
\end{equation}
Here, $B_k(x)$ are the ordinary Bernoulli polynomials
and $B_k^{(2)}(x)$
denote the Bernoulli polynomials of order 2 and degree $k$ 
defined by
\begin{equation*}
    B_k^{(2)}(x)=\sum_{\alpha=0}^k\binom{k}{\alpha}
    B_\alpha^{(2)}\,x^{k-\alpha}.
\end{equation*}
This gives
\begin{equation}\label{eq:Klm5}
\begin{split}
    K_{l,m}
    &=\frac{1}{2N}
    \biggl[
    B_2^{(2)}(1)+\Bigl(\Bigl(\frac{l-m}{N}\Bigr)^2-\Bigl(\frac{l-m}{N}\Bigr)+\frac{1}{6}\Bigr)
    B_0^{(2)}(1)N^2\biggr]\\
    &=\frac{1}{2N}
    \biggl[
    B_0^{(2)}+2B_1^{(2)}+B_2^{(2)}+\Bigl(\Bigl(\frac{l-m}{N}\Bigr)^2-\Bigl(\frac{l-m}{N}\Bigr)+\frac{1}{6}\Bigr)
    B_0^{(2)}N^2\biggr].
\end{split}
\end{equation} 
$B_k^{(2)}$ denote the Bernoulli numbers of order 2 and degree $k$ given by the generating function \cite{carlitz1952some}
\begin{equation*}
  \sum_{k=0}^\infty\frac{t^k}{k!}B_k^{(2)}
  = \Bigl(\frac{t}{e^t-1}\Bigr)^2
  = 1-t+ \frac{5}{12} t^2 -\frac{1}{12} t^3+O(t^4).
\end{equation*}
Thus, we obtain
\begin{equation}\label{eq:Klm7}
\begin{split}
       K_{l,m}
    &=\frac{1}{2N}
    \biggl[
    1-2+\frac{5}{6}+\Bigl(\Bigl(\frac{l-m}{N}\Bigr)^2-\Bigl(\frac{l-m}{N}\Bigr)+\frac{1}{6}\Bigr)
    N^2\biggr]\\
    &=\frac{1}{2N}
    \biggl[
    -\frac{1}{6}+\Bigl(\Bigl(\frac{l-m}{N}\Bigr)^2-\Bigl(\frac{l-m}{N}\Bigr)+\frac{1}{6}\Bigr)
    N^2\biggr]\\
    &=\frac{1}{2N}
    \biggl[
    (l-m)^2-N(l-m)+ \frac{N^2-1}{6}\biggr].
\end{split} 
\end{equation}
This yields \eqref{eq:K_final} and completes the proof.
\end{proof}

Theorem \ref{thm:limitX} and Proposition \ref{prop:exp_K} show that the limit distribution of $(X(t))_{t\ge 0}$ is Gaussian with expectation vector and covariance matrix 
\begin{equation*}
\mu_X(\infty)=
\begin{bmatrix}
\frac{L}{N}\mathbf{1}\\
    0
\end{bmatrix} \in \R^{2N}, \qquad
\Sigma_X(\infty)=
\begin{bmatrix}
    \frac{\sigma^2}{2\alpha^2 \beta }K & 0\\
    0 & \frac{\sigma^2}{2 \beta }K
\end{bmatrix}\in \R^{2N\times 2N},
\end{equation*}
where the matrix $K=K(N)$ only depends on $N$ and satisfies
\begin{equation}\label{eq:K_final2}
    K_{l,m}=\frac{1}{2N}
    \biggl[
    (l-m)^2-N|l-m|+ \frac{N^2-1}{6}\biggr].
\end{equation}
We denote by $X(\infty)=(Q(\infty),D(\infty))^\top \in\R^{2N}$ a random vector with this distribution. 
Then we see that in the limit $t\to \infty$
\begin{itemize}
    \item the agents' positions are distributed equidistantly in expectation,
    \item the deviations $D_n(\infty)$, $n\in \{1,\ldots,N\}$, from the ensemble's mean velocity are zero in expectation,
    \item the covariance of deviations $D_n(\infty)$, $n\in \{1,\ldots,N\}$, from the ensemble's mean velocity are independent of $\alpha$ and proportional to $\frac{1}{\beta}$,
    \item the covariance of the distances $Q_n(\infty)$, $n\in \{1,\ldots,N\}$, is equal to $\frac{1}{\alpha^2}$ times the covariance of the deviations from the mean velocity and
    \item the covariances between distances $Q_n(\infty)$, $n\in \{1,\ldots,N\}$, and velocity deviations $D_n(\infty)$, $n\in \{1,\ldots,N\}$, are zero, and hence independent (since they are Gaussian).
\end{itemize}

The matrix $K$ is circulant. This reflects the interchangeability of the agents and the symmetry of the problem. To further discuss the implications of \eqref{eq:K_final2}, we take the first agent as representative agent and only consider the first row of $K$. Let $c(x)=\frac{1}{2N}
  [(x-1)^2-N(x-1)+ \frac{N^2-1}{6}]$, $x\in \R$. Then $(c(m))_{m=1}^{N}$ describes the first row of $K$ and thus the block covariance matrices in \eqref{eq:K_final2}. In particular, the vector $(c(m))_{m=1}^{N}$ specifies up to the factor $\frac{\sigma^2}{2\alpha^2 \beta}$ the covariances between the distance $Q_1(\infty)$ between the first and the second agent and the distances $Q_n(\infty)$, $n\in \{1,\ldots,N\}$, between the other neighboring agents. Moreover, it also describes up to the factor $\frac{\sigma^2}{2 \beta}$ the covariances between the first agent's deviation $D_1(\infty)$ from the ensemble's mean velocity with the velocity deviations $D_n(\infty)$, $n\in \{1,\ldots,N\}$, of the other agents.
  Note that 
    $c(x)=\frac{1}{2N}
  [(x-(\frac{N}{2}+1))^2-\frac{N^2+2}{12}]$. So $c$ describes a convex parabola with vertex at $(\frac{N}{2}+1,-\frac{N^2+2}{24N})$. In particular, $c(1)=\frac{N^2-1}{12N}$ -- the variance of the first agent's coordinates -- is the maximal entry of $(c(m))_{m=1}^{N}$. As $m$ increases, $c(m)$ first decreases, crosses zero before it reaches its minimal negative value at $m=\lfloor \frac{N}{2}+1\rfloor$ and $m=\lceil \frac{N}{2}+1\rceil$ (which are the same if $N$ is even).   
  Note that agents $m=\lfloor \frac{N}{2}+1\rfloor$ and $m=\lceil \frac{N}{2}+1\rceil$ are the agents with the largest distance to the first agent. 
  As $m$ further increases, $c(m)$ increases as well.
  This discussion implies that in the steady state there is with high probability maximally one wave/cluster of agents. For example, if $Q_1$ is smaller than the expectation $L/N$, then the immediate neighbors (with large probability) 
also have a small distances. 
The more distant agents (in particular $m=\lfloor \frac{N}{2}+1\rfloor$ and $m=\lceil \frac{N}{2}+1\rceil$) have negative correlation and thus large distances with higher probability. This means roughly speaking that the agents agglomerate around the first agent.
A similar observation applies to the velocity coordinates. The entries of $K$ are depicted in Figure~\ref{fig:0}.

\begin{figure}
    \centering
    \includegraphics{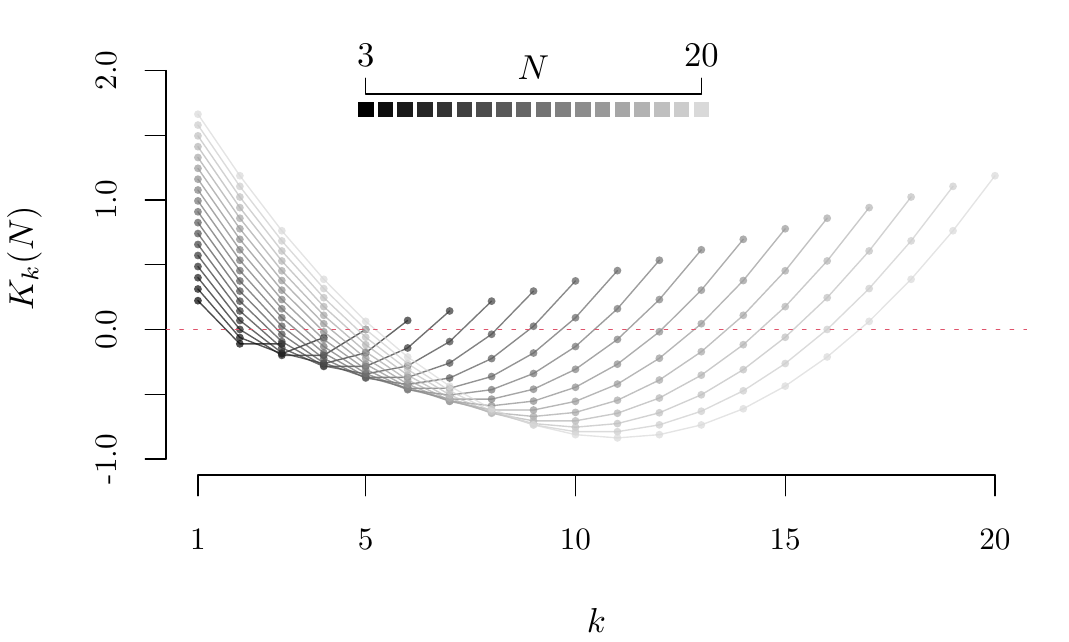}
    \caption{Plot of the function $K(N):k\mapsto K_k(N)=\frac{1}{2N}
    \biggl[
    (k-1)^2-N(k-1)+ \frac{N^2-1}{6}\biggr]$ 
    for $k\in\{1,\ldots,N\}$ and $N\in\{3,4,\ldots,20\}$.
    }
    \label{fig:0}
\end{figure}

\section{Numerical Experiments}\label{sec:simulation}

We present in this section some simulation results of agents on a segment with periodic boundaries. The code can be downloaded and the simulations can be computed in real time on the online platform at \href{https://www.vzu.uni-wuppertal.de/fileadmin/site/vzu/Simulating_Collective_Motion.html?speed=0.7}{\texttt{https://www.vzu.uni-wuppertal.de/fileadmin/site/vzu/Simulating\_Collective\_Motion.html}}.

\subsection{Simulation Setup}
We simulate the trajectories of $N=20$ agents on a segment of length $L=501$ with periodic boundary conditions. 
The initial condition is uniform with velocity zero, i.e.\ 
\begin{equation*}
    Q_n(0)=L/N\quad\text{and}\quad p_n(0)=0,\qquad n=\{0,\ldots,N\}.
\end{equation*}

\paragraph{Numerical solver}
The simulations are computed using an implicit/explicit Euler-Maruyama scheme with time step $\delta t$.  
We denote in the following the system actualisation by $t_k=k\delta t$, $k\in\{0,1,2,\ldots\}$.  
The numerical scheme for the $n$-th agent at time $t_k$ is given by
\begin{equation}\label{eq:num_modn}
    \begin{cases}
        q_n(t_{k+1})&=q_n(t_k)+\delta t\, p_n(t_{k+1}),\\
        p_n(t_{k+1})&=p_n(t_k)+\bigl(U'(Q_n(t_k))-U'(Q_{n-1}(t_k))\bigr)\,\delta t\\
        &\qquad+\;\beta\bigl(p_{n+1}(t_k)-2p_n(t_k)+p_{n-1}(t_k)\bigr)\,\delta t+\sigma\xi_n(k)\,\sqrt{\delta t}, 
    \end{cases}
\end{equation}
with $(\xi_n(k))_{n=0}^N$, $k\in\{0,1,2,\ldots\}$, independent one-dimensional standard normal random variables. 
Note that different simulation schemes for deterministic port-Hamiltonian pedestrian models are compared in \cite{tordeux2022multi}. 
The implicit/explicit Euler schemes prove to be efficient solvers.
In the following, we set the time step to $\delta t=0.001$, which seems to be a good compromise between accurate numerical approximations and reasonable run times.

\paragraph{Parameters' setting} 
The parameters' settings are as follows. 
The dissipation rate and the noise volatility are set to one
\begin{equation*}
  \beta=1\qquad\text{and}\qquad\sigma=1, 
\end{equation*}
while the potential $U\colon \R \to [0,\infty)$ is given by
\begin{equation*}
    U(x)=\frac{1}{\kappa} (\alpha |x|)^\kappa, \qquad x\in \R
\end{equation*}
for some fixed $\alpha\in (0,\infty)$ and $\kappa\in (1,\infty)$.
Note that $U$ is a convex function with derivative $U'(x)=\sgn(x)\alpha^\kappa |x|^{\kappa-1}$.

\paragraph{Simulation scenario}
In the motion model \eqref{eq:modn}, the parameters $\alpha$ and $\kappa$ control the distances while the parameter $\beta$ controls the relative velocities with the neighbors.
Three simulation scenarios are examined in the following,
focusing on the role of the parameters $\alpha$ and $\kappa$. 
We consider the quadratic potential for which $\kappa=2$ with $\alpha=0.1$ and $\alpha=1$ for the first two scenarios, respectively.
The system in these cases is the linear Ornstein-Uhlenbeck process described in Section~\ref{sec:OUprocess}. 
We further analyze the case of a non-linear distance-based interaction term with $\kappa=4$ and $\alpha=1$ in a final simulation scenario. 
\medskip

In the following, we focus on the emergence of collective motions in the three scenarios before analyzing the autocorrelation and the distributions of the ensemble's velocity and distance variances for long simulation times.

\subsection{Simulating Collective Motion}

Thanks to the telescopic form of model, the ensemble mean velocity of the agents $(\xbar{p}(t))_{t\ge 0}$ is a Brownian motion with variance $\sigma^2/N$, see Remark~\ref{rem:average_velocity}. 
As time progresses, the ensemble mean velocity fluctuates. 
In parallel, the distributions of the ensemble variances of the agents' velocities and agents' distances converge to finite settings, see Theorem~\ref{thm:limitX}.
The divergence of the ensemble mean velocity coupled to the convergence of ensemble variances of the agents' velocities and agents' distances trigger the agents to move in a coordinated manner.
Note that, the deterministic system with $\sigma=0$ is stable and systematically converges to an uniform equilibrium solution for which $p_n(\infty)=\xbar{p}_0$ and $Q_n(\infty)=L/N$ for all $n\in\{1,\ldots,N\}$. 
The coordination of the dynamics is purely noise-induced and would vanish if the stochastic perturbations disappear. 

The next figures present the evolution of the system in the three scenarios over the first $K_\text{max}=500\,000$ simulation steps. 
The figures consist of three panels:
\begin{itemize}
    \item Top panel: the agents' trajectories $(q_n(t_k))_{n=1}^N$ for $k\in \{1,\ldots, K_{max}\}$.
    \item Central panel: the velocity of the first agent $p_1(t_k)$, the mean velocity
    \begin{equation*}
        \xbar{p}(t_k)=\frac{1}{N}\sum_{n=1}^N p_n(t_k),
    \end{equation*}  
    and the ensemble variance of velocities
    \begin{equation*}
        V_p(t_k)=\frac{1}{N}\sum_{n=1}^N \bigl(p_n(t_k)-\xbar{p}(t_k)\bigr)^2=\frac{1}{N}\|Mp(t_k)\|^2=\frac{1}{N}\|D(t_k)\|^2,
    \end{equation*}
    for $k\in \{1,\ldots, K_{max}\}$. 
    $M$ being the matrix introduced in \eqref{eq:def_M} and $D$ the process defined in \eqref{eq:def_D}.
    \item Bottom panel: the distance of the first agent $Q_1(t_k)$ minus the mean distance $L/N$ and the ensemble variance of distances 
    \begin{equation*}
        V_Q(t_k)=\frac{1}{N}\sum_{n=1}^N \bigl(Q_n(t_k)-L/N\bigr)^2=\frac{1}{N}\|MQ(t_k)\|^2=\frac{1}{N}\|E(t_k)\|^2
    \end{equation*}   
    for $k\in \{1,\dots, K_{max}\}$. 
    Here we introduced the deviation from the mean distance process $E=MQ(t)$, $t\ge 0$.
\end{itemize} 
The same random kernel is used for all three simulation scenarios. 
Since the mean velocity is independent of the $\alpha$, $\kappa$, and $\beta$ parameters, it can be used as an invariant reference for the scenarios. 

The ensemble mean velocity describes a random walk, tending to be positive for the first simulation times (i.e., up to approximately $k=200\,000$) before fluctuating towards negative values (see Figures~\ref{fig:1}, \ref{fig:2}, or \ref{fig:3}, the black curve in the middle panel). 
Correspondingly, the agents appear to move collectively to the left before moving in a coordinated manner to the right (see the agent trajectories in Figures~\ref{fig:1}, \ref{fig:2}, and \ref{fig:3}, top panels, the blue trajectory is that of the first vehicle). 
However, the distances are distended and fluctuating in the first scenario with $\alpha=0.1$ and $\kappa=2$. 
The trajectories allow for the appearance of jamming waves and large distance fluctuations (see Figure~\ref{fig:1}, top panel). 

\begin{figure}[!ht]
    \centering\medskip
    \includegraphics[width=\textwidth]{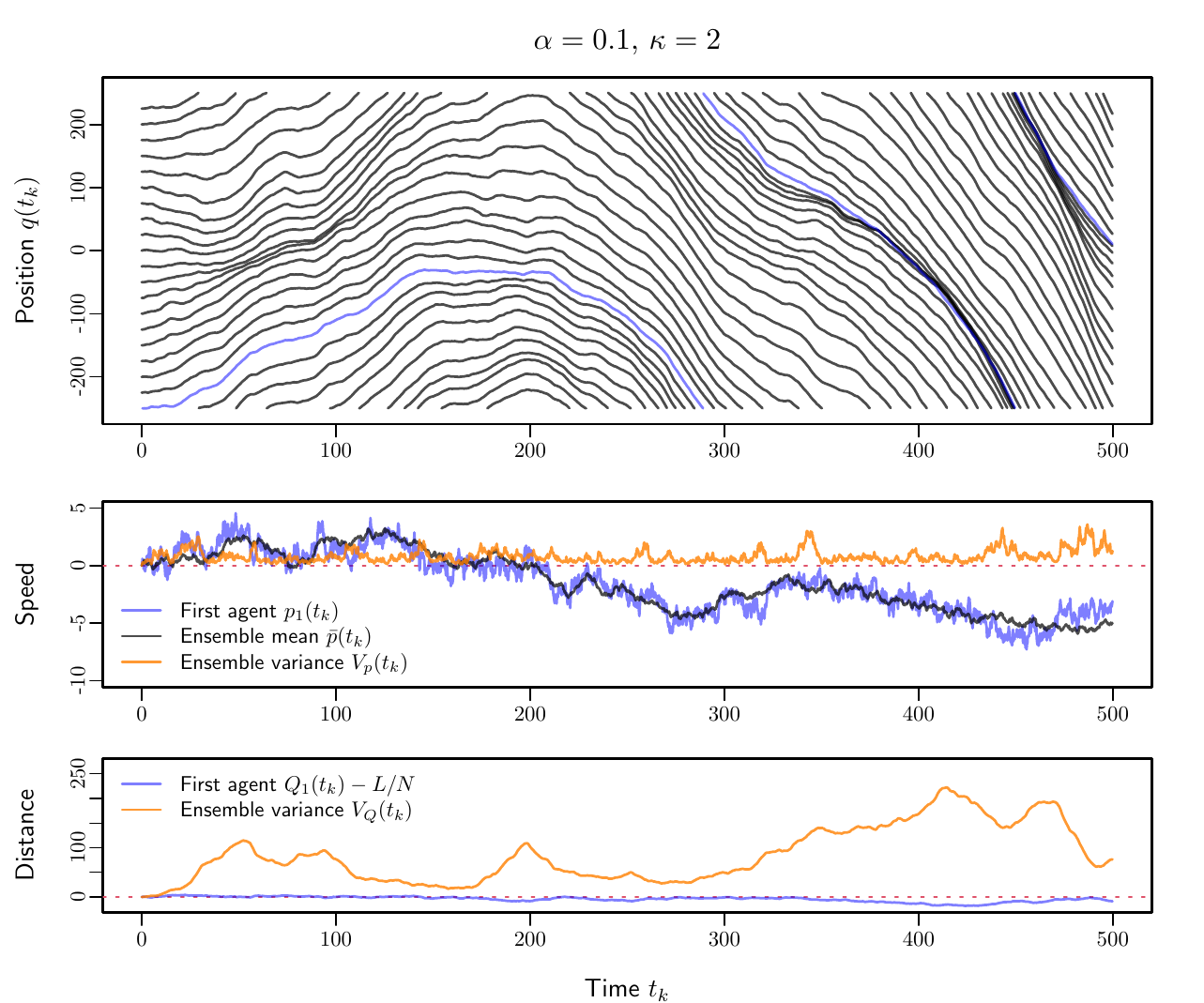}\medskip
    \caption{First simulation scenario with $\alpha=0.1$ and $\kappa=2$ (Ornstein-Uhlenbeck process). Top panel: Agents' trajectories. Middle panel: Agents' velocity features. Bottom panel: Agents' distance features. The agents appear to move to the left before moving in a coordinated manner to the right (top panel). Indeed, the mean velocity, being a Brownian motion, fluctuates towards positive values before visiting negative values (middle panel). The ensemble variance of the agents' velocities seems to become quickly stationary. However, the dynamics are disordered and even show jamming waves. The ensemble variance of the agents' distance presents large fluctuations (bottom panel).}
    \label{fig:1}
\end{figure}

In contrast, the trajectories appear more uniform in the second scenario where the parameter $\alpha=1$, accounting for the distance, is larger (see  Figure~\ref{fig:2}, top panel). 
The trajectories are again more regular in the third simulation scenario with $\alpha=1$ and $\kappa=4$. 
The system is no longer an Ornstein-Uhlenbeck process in this scenario since the distance-based interaction term $U'(Q)=\alpha^4Q^3$ is no longer linear. 
The hard regulation of the distance induced by the nonlinearity of the interaction term makes the trajectories almost equi-distant, although fluctuating according to the Brownian motion of the mean velocity. 
Indeed, the ensemble variance of the agents' distance presents large fluctuations in the first scenario, while it is reduced in the second scenario and close to zero in the third scenario (see, respectively, Figures~\ref{fig:1}, \ref{fig:2}, and \ref{fig:3}, bottom panels).

\begin{figure}[!ht]
    \centering\medskip
    \includegraphics[width=\textwidth]{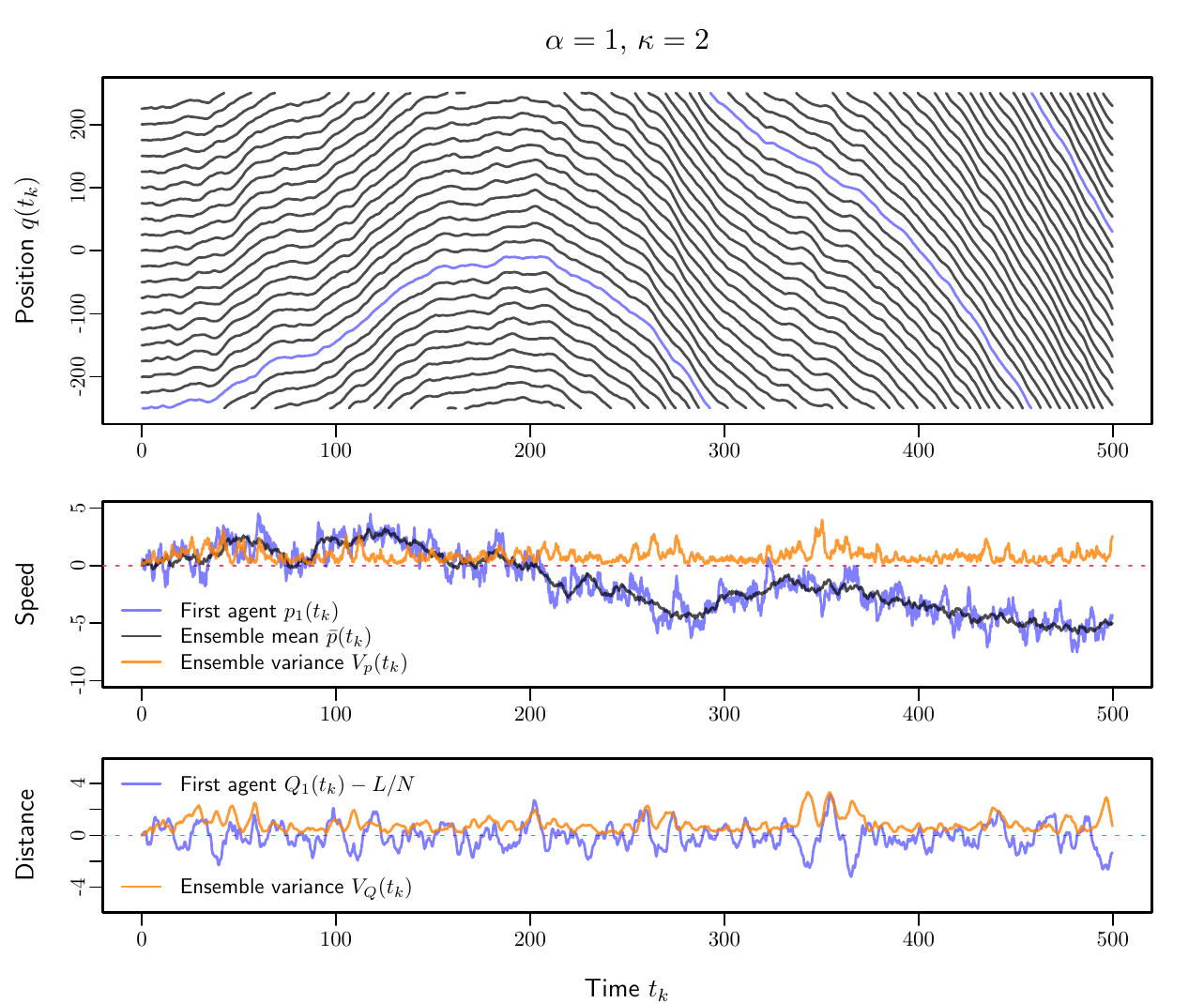}\medskip
    \caption{Second simulation scenario with $\alpha=1$ and $\kappa=2$ (Ornstein-Uhlenbeck process). Top panel: Agents' trajectories. Middle panel: Agents' velocity features. Bottom panel: Agents' distance features. The collective motion of the agents is more uniform when $\alpha=1$ than when $\alpha=0.1$ (compare with the first scenario Figure~\protect\ref{fig:1}). The variability of the distance is much more reduced for $\alpha=1$, but the variability of the velocity shows qualitatively comparable characteristics.}
    \label{fig:2}
\end{figure}

\begin{figure}[!ht]
    \centering\medskip
    \includegraphics[width=\textwidth]{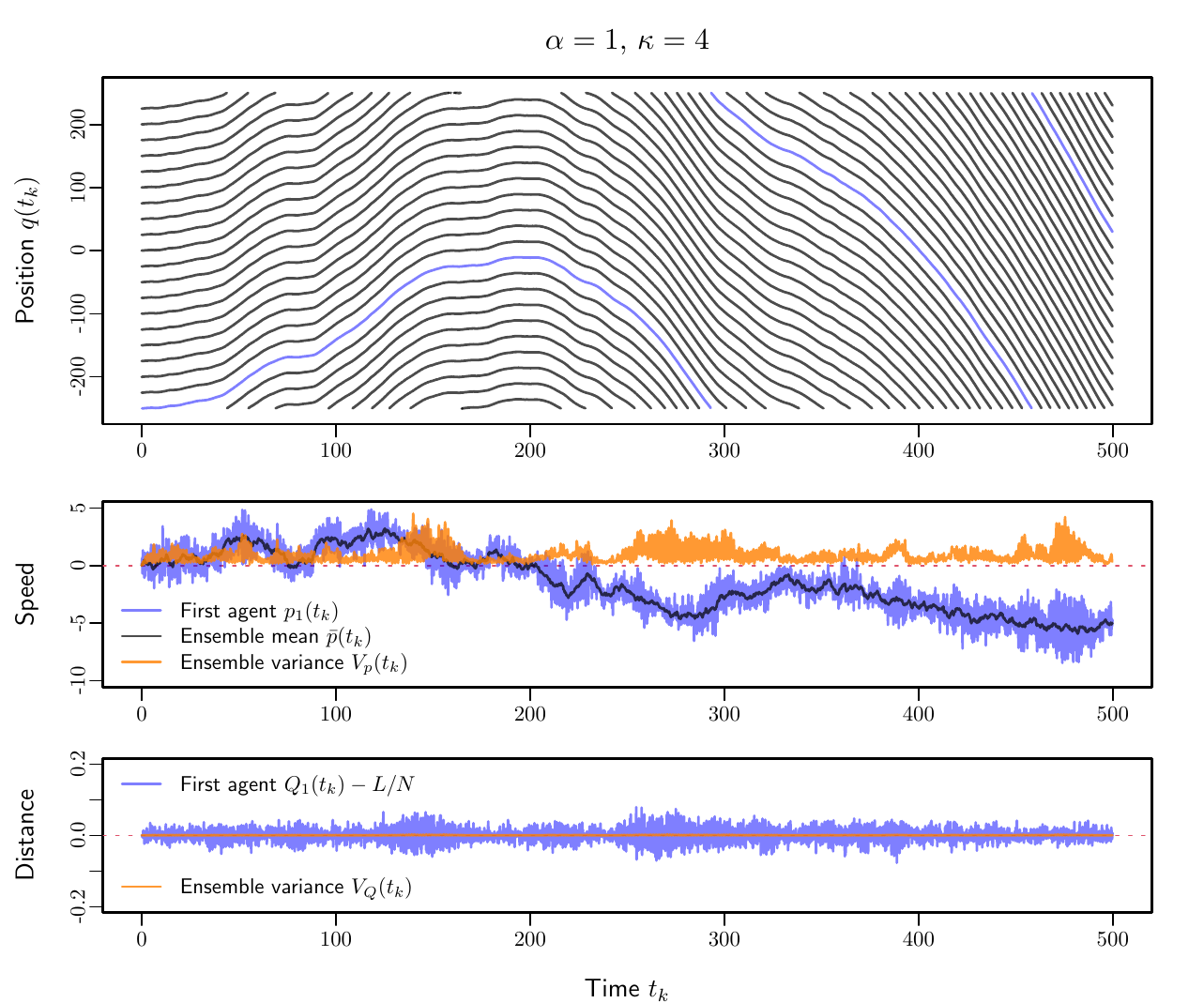}\medskip
    \caption{Third simulation scenario with $\alpha=1$ and $\kappa=4$ (non-linear model). Top panel: Agents' trajectories. Middle panel: Agents' velocity features. Bottom panel: Agents' distance features. The collective motion of the agents is more regular again in the case of the hard (non-linear) distance-based interaction term (compare with the first and second scenario Figures~\protect\ref{fig:1} and \protect\ref{fig:2}). The distance's variability is close to zero, and the agents move synchronously according to the Brownian motion of the mean velocity. Interestingly, even in the non-linear case, the velocity variability shows similar distribution characteristics to those obtained with the linear model.}
    \label{fig:3}
\end{figure}

\subsection{Autocorrelations of the Ensemble Variances}

The trajectories of the velocity and distance ensemble variances 
with $\alpha=0.01$ is more regular than the trajectories with $\alpha=1$ in the first and second linear scenarios for which $\kappa=2$. 
The trajectories with the nonlinear model with $\alpha=1$ and $\kappa=4$ are much less regular again (see Figure~\ref{fig:3}, middle and bottom panels). 
The Figure~\ref{fig:4} presents the empirical autocorrelation for the ensemble variances of the agents' velocities and distances for the three scenarios. 
The estimation are obtained based on simulation histories of $K_\text{max}=10\,000\,000$ iterations. 
A rich variety of dynamics can be observed. 
The first scenario with $\alpha=0.1$ present overdamped features for the autocorrelation function (gray curves in Figure~\ref{fig:4}). 
In fact, it holds for this scenario
\begin{equation*}
    \alpha^2<\frac12\beta^2(1-\cos(2\pi j/N)),\qquad j\in\{1,\ldots,N-1\},
\end{equation*}
($N=20$, $\beta=1$ and $\frac12\beta^2(1-\cos(2\pi j/N))\approx0.02$) and the system eigenvalues \eqref{eq:eigenvalue} are all purely real (see also Remark \ref{rem:disc_ev}), making the dynamics oscillation-free. 
The overdamped stability does not hold for $\alpha=1$ (second scenario). 
The eigenvalues are complex numbers and the velocity and distance ensemble variances describe damped oscillatory behaviors (blue curves in Figure~\ref{fig:4}). 
For the nonlinear model with $\kappa=4$, we observe extremely reduced oscillations (third scenario, orange curves in Figure~\ref{fig:4}). 
Indeed, the trajectories show deterministic features, even if they collectively fluctuate according to the Brownian motion of the mean velocity. 
This reduced oscillatory behavior explains why the ensemble variance trajectories present irregular characteristics (see Figure~\ref{fig:3}, bottom and middle panels).

\begin{figure}[!ht]
    \centering\medskip\medskip
    \includegraphics[width=\textwidth]{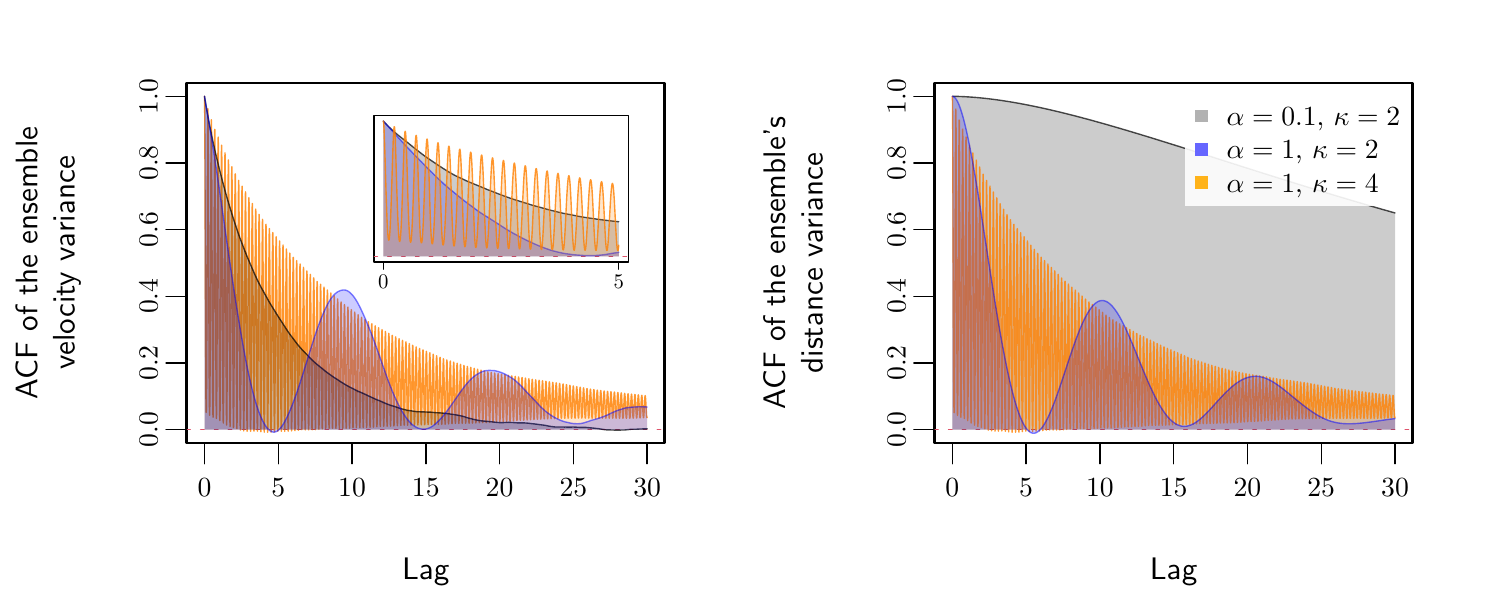}
    \caption{Autocorrelation functions (ACF) for the ensemble variances of velocities (left panel) and distances (right panel). The dynamics being overdamped in the first scenario where $\alpha=0.01$ and $\kappa=2$, the ACFs show no oscillation. The dynamics are solely damped when $\alpha=1$ and $\kappa=2$ and the ACF present oscillations with a period close to 10. For the nonlinear model with $\alpha=1$ and $\kappa=4$, the dynamics are  again much more oscillatory (period close to 0.2).}
    \label{fig:4}
\end{figure}

\subsection{Distributions of the Ensemble Variances}

In this section, we analyze the distribution of the ensemble variances of agents' velocities and distances for long simulation times. 
For the quadratic potentials $U$ with $\kappa=2$, the random variables $D(t_k)$
and $E(t_k)$ are normally distributed for every $k\in \N$. This implies that for every $k\in \N$ the ensemble variance of velocities $V_p(t_k)=\frac{1}{N}\|D(t_k)\|^2$ and the ensemble variance of distances $V_Q(t_k)=\frac{1}{N}\|E(t_k)\|^2$ follow generalized chi-squared distributions. For large $k\in \N$ the random variables $D(t_k)$ and $E(t_k)$ are asymptotically centered normal with covariance
\begin{equation}\label{eq:covTW}
    \Sigma_D(\infty)=\frac{\sigma^2}{2 \beta }K
    \qquad\text{and}\qquad
    \Sigma_E(\infty)=\frac{\sigma^2}{2\alpha^2 \beta }K, 
\end{equation}
respectively (see Theorem~\ref{thm:limitX}). Moreover, they are asymptotically independent. This implies that for large $k\in \N$ the ensemble variances $V_p(t_k)$ and $V_Q(t_k)$ have asymptotically independent chi-squared distributions whose parameters are determined through \eqref{eq:covTW}. 
Note that the distribution for the ensemble's velocity variance is independent of $\alpha$.

In particular, still in the case $\kappa=2$, the asymptotic expected ensemble variance of velocities is given by
\begin{equation*}
    \E\bigl(V_p(\infty)\bigr)=\frac{\sigma^2}{2N\beta }\sum_{n=1}^NK_{n,n}= \frac{\sigma^2(N^2-1)}{24N\beta}=0.83125
\end{equation*}
as $K_{n,n}=(N^2-1)/(12N)$ for all $n\in\{1,\ldots,N\}$, see \eqref{eq:K_final}, $N=20$, and $\beta=\sigma=1$. 
Similarly, the asymptotic expected ensemble variance of distances reads
\begin{equation*}
    \E\bigl(V_Q(\infty)\bigr)=\frac{\sigma^2(N^2-1)}{24N\alpha^2\beta}=\frac{0.83125}{\alpha^2}.
\end{equation*}

The trajectories of the ensemble variances of the agents' velocities and distances exhibit different characteristics. 
However, similar ranges of variation for the velocity variance appear for the three scenarios including the nonlinear model with $\kappa=4$ as well (third scenario), see  Figures~\ref{fig:1}, \ref{fig:2}, and \ref{fig:3}, middle panels.
We simulate the three scenarios over $K_{\text{max}}=200\,000\,000$ iterations and collect every $50\,000$ iterations the ensemble variances of agents' velocities and distances (samples of $4\,000$ observations). 
We estimate the theoretical chi-squared distribution using Monte Carlo simulation of random variables $\|C\mathcal U\|^2/N$, with $\mathcal U\colon \Omega \to \R^N$ a random vector of independent standard normal random variables and $C\in\R^{N\times N}$ the Cholesky decomposition of the covariance matrix $\Sigma_D(\infty)$ and $\Sigma_E(\infty)$, respectively, see \eqref{eq:covTW} (i.e., $CC^\top=\Sigma_D(\infty)$ for the chi-squared distribution of the ensemble's velocity variance while $CC^\top=\Sigma_E(\infty)$ for the ensemble's distance variance).
The histograms of simulated measurements and the theoretical chi-squared distributions are presented in Figure~\ref{fig:5}. 
Interestingly, it turns out that the chi-squared distribution of the agents' asymptotic ensemble's velocity variance also fits the histograms of the simulation of the nonlinear model with $\kappa=4$ (third scenario).  
This is surprising since the system is no longer an Ornstein-Uhlenbeck process if $\kappa=4$. 
Different distance-based interaction terms (e.g., with $\kappa=6$ or $U(x)=\exp(\alpha x)$) presented the same characteristic in further numerical experiments.

We have thus seen that the distribution of the agents' asymptotic ensemble's velocity variance does not display a dependence on $\alpha$ and $\kappa$. 
In contrast, the asymptotic distribution of the ensemble variance of the agents' distances is directly impacted by these parameters.
The ensemble variance of the agents' distances is asymptotically proportional to the inverse of the square of the parameter $\alpha$ in the linear case for which $\kappa=2$ (see Theorem \ref{thm:limitX} and Figure~\ref{fig:5}, bottom panels). 
The ensemble variance of the distances is thus on average a factor of 100 larger in the first scenario than in the second.
In the third scenario, where the distance-based interaction term is cubic, the distance variance is even more reduced and close to zero.
This is not surprising, since the non-linearity of the interaction term makes the model extremely sensitive to fluctuations in the distances to the nearest-neighbors. 
As a result, the distribution of the agents in space becomes increasingly uniform in the scenarios 1, 2, and 3 (see the trajectories in Figures~\ref{fig:1}, \ref{fig:2}, and \ref{fig:3}, top panels). 

\begin{figure}[!ht]
    \centering\medskip\medskip
    \includegraphics[width=\textwidth]{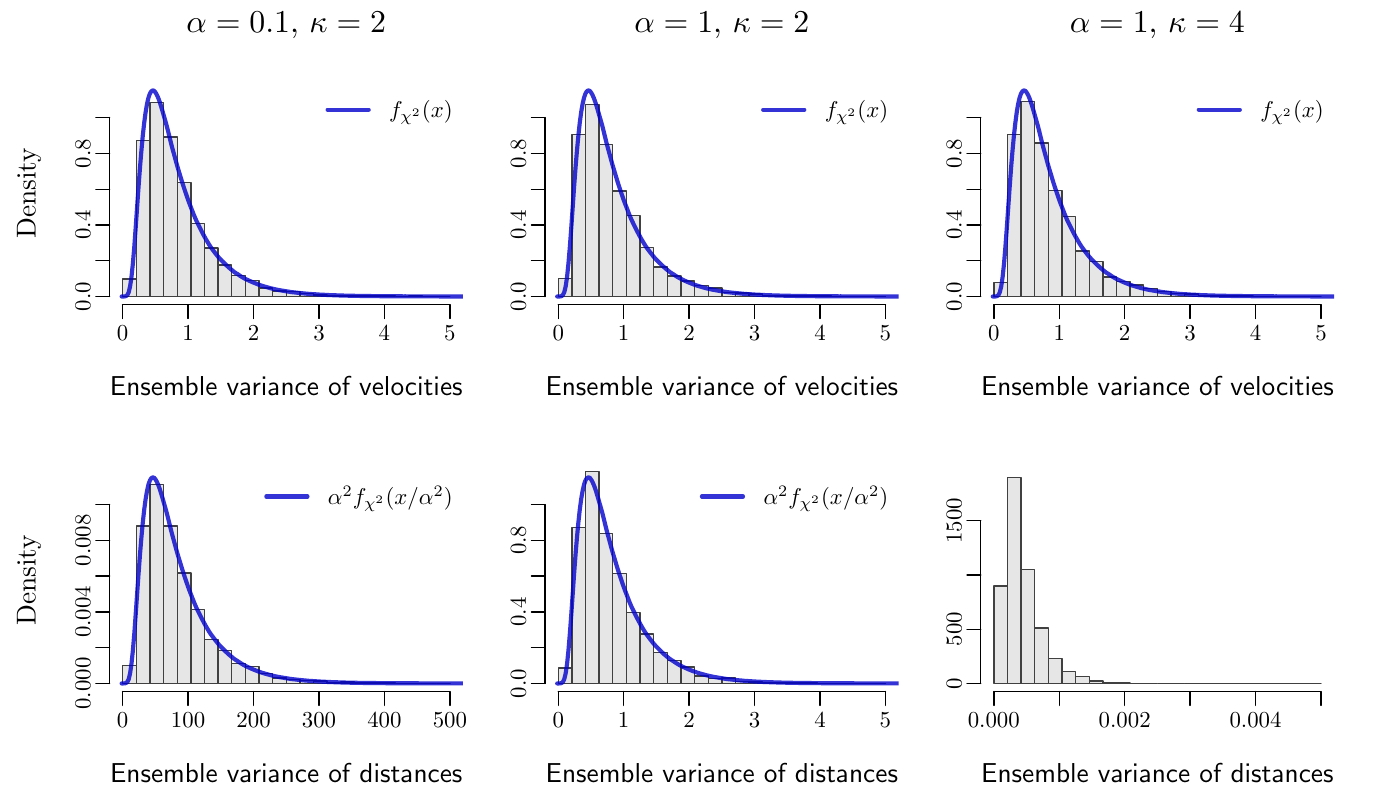}\medskip
    \caption{Histograms of the ensemble variance of the agents' velocities (top panels) and the ensemble variance of the agents' distances (bottom panels) obtained by simulation and the theoretical generalized chi-squared distributions for, from the left to the right, the first, second and third scenarios. As expected, the theoretical results match the simulation in the first and second scenario with $\kappa=2$ (left and middle panels). It is interesting to note that the ensemble variance of velocities for the nonlinear model also appears to follow a chi-squared distribution (third scenario with $\kappa=4$, top right panel).}
    \label{fig:5}
\end{figure}

\paragraph{Acknowledgments}
The authors thank Barbara R{\"u}diger, Claudia Totzeck and Baris Ugurcan for fruitful discussions on port-Hamiltonian interacting particle systems.

\bibliographystyle{plain} 
\bibliography{refs} 


\end{document}